\let\ORIlabel\label
\let\ORIrefstepcounter\refstepcounter
\AddToHook{package/hyperref/before}{
   \let\label\ORIlabel 
   \let\refstepcounter\ORIrefstepcounter}
\documentclass[review,hidelinks,onefignum,onetabnum]{siamart251216}

\newcounter{tablepanel}

\usepackage{graphicx,amsmath,amssymb}
\usepackage{algorithm}
\usepackage{tabularx}
\usepackage{enumitem}
\usepackage{algpseudocode}
\usepackage{booktabs}
\usepackage{threeparttable}
\usepackage{makecell}
\usepackage{epstopdf}
\usepackage{float}
\usepackage{appendix}
\usepackage{tikz}
\usepackage{cases}
\usepackage{mathrsfs}
\usepackage{bm} 
\usepackage{mathtools}
\DeclarePairedDelimiterX{\inner}[2]{\langle}{\rangle}{#1,\,#2}

\ifpdf
  \DeclareGraphicsExtensions{.eps,.pdf,.png,.jpg}
\else
  \DeclareGraphicsExtensions{.eps}
\fi

\newsiamremark{remark}{Remark}
\newsiamremark{hypothesis}{Hypothesis}
\crefname{hypothesis}{Hypothesis}{Hypotheses}
\newsiamthm{assumption}{Assumption}

\headers{Preconditioned High-index Saddle Dynamics}{B. Huang, H. Su, L. Zhang, J. Zhao}

\title{Computing Saddle Points in Stiff Problems via a Preconditioned High-index Saddle Dynamics Method \thanks{Submitted to the editors DATE.
\funding{L.Z. was supported by the National Natural Science Foundation of China (No. 12225102, T2321001, and 12288101) and the National Key Research and Development Program of China  2024YFA0919500. J.Z. was supported in part by the Beijing Natural Science Foundation (Grant No. JR25003), the National Natural Science Foundation of China (Grant No.12301520), and the Beijing Outstanding Young Scientist Program (No.JWZQ20240101027).}}}

\author{Bingzhang Huang\thanks{School of Mathematical Sciences, Peking University, Beijing, 100871, China (\email{1160205334a@gmail.com}).}
\and Hua Su\thanks{Beijing International Center for Mathematical Research, Peking University, Beijing, 100871, China. (\email{suhua@pku.edu.cn}).}
\and Lei Zhang\thanks{Corresponding author. School of Mathematical Sciences, Beijing International Center for Mathematical Research, Center for Quantitative Biology, Center for Machine Learning Research, Peking University, Beijing 100871, China (zhangl@math.pku.edu.cn).}
\and Jin Zhao\thanks{Academy for Multidisciplinary Studies, Capital Normal University, Beijing, 100048, China.
(\email{zjin@cnu.edu.cn}).}
}

\usepackage{amsopn}
\DeclareMathOperator{\diag}{diag}

\ifpdf
\hypersetup{
  pdftitle={Preconditioned High-index Saddle Dynamics for Computing Saddle Points},
  pdfauthor={Bingzhang Huang, Hua Su, Lei Zhang, Jin Zhao}
}
\fi

\begin{document}

\maketitle

\begin{abstract}
High-index saddle dynamics (HiSD) is an effective approach for computing saddle points of a prescribed Morse index and constructing solution landscapes for complex nonlinear systems. However, for problems with ill-conditioned Hessians arising from fine discretizations or stiff potentials, the efficiency of standard HiSD deteriorates as its convergence rate worsens with the spectral condition number $\kappa$. To address this issue, we propose a preconditioned HiSD (p-HiSD) framework that reformulates the continuous dynamics within a Riemannian metric induced by a symmetric positive definite preconditioner $M$. 
By generalizing orthogonal reflections and unstable-subspace tracking to the $M$-inner product, the proposed scheme modifies the geometry of the saddle-search dynamics while remaining computationally efficient.
Rigorous theoretical analysis confirms that the equilibria and their Morse indices are invariant under this metric. Furthermore, we establish the local exponential stability of the continuous dynamics and prove a discrete linear convergence rate governed by the preconditioned condition number $\kappa_M$. Consequently, the iteration complexity is sharply reduced from $O(\kappa\log(1/\epsilon))$ to $O(\kappa_M\log(1/\epsilon))$. We validate the method on nine numerical tests spanning finite-dimensional model problems, stiff lattice systems, and PDE discretizations. The results demonstrate that p-HiSD resolves stiffness-induced convergence failures, permits substantially larger step sizes, and significantly reduces iteration counts.
\end{abstract}



\section{Introduction}
Saddle points of energy functionals arise in many areas of science and engineering, including phase transitions in materials science~\cite{doi:10.1016/j.actamat.2009.10.041,PhysRevLett.98.265703}, nucleation phenomena in soft matter~\cite{PhysRevLett.104.148301}, transition pathways in chemical reactions~\cite{E2010TransitionpathTA,doi:10.1126/science.1253810}, protein folding~\cite{doi:10.1038/nsb0497-305}, and the loss landscape of deep neural networks~\cite{Dauphin2014IdentifyingAA,Fukumizu2019SemiflatMA}. Unlike local minima, which can be found by gradient descent, saddle points of general index require dedicated algorithms that can simultaneously ascend along unstable directions and descend along stable ones.

Over the past two decades, several computational methods have been developed for saddle point search in complex landscapes~\cite{bonfanti2017methods}, including the dimer method~\cite{Gould2016dimertype,doi:10.1063/1.480097}, the gentlest ascent dynamics and its variants~\cite{QUAPP2015,BOFILL2013203,E2011gentlest,gu2018simplified}, minimax methods~\cite{Li2001minimax,li2002convergence}, and the string method~\cite{doi:10.1063/1.2013256}, and recent proximal-minimization approaches \cite{gu2025iterative}. Among these, the high-index saddle dynamics (HiSD)~\cite{Yin2019HiSD} stands out for its ability to systematically compute saddle points of prescribed Morse index $k$ by coupling a position update---which ascends along the $k$-dimensional unstable subspace and descends along its complement---with a frame dynamics that tracks the unstable eigenvectors via Rayleigh quotient minimization on the Stiefel manifold. HiSD has been successfully applied to construct solution landscapes~\cite{PhysRevLett.124.090601,Yin2020} for liquid crystals~\cite{doi:10.1088/1361-6544/abc5d4,doi:10.1017/S0962492921000088}, quasicrystals~\cite{Yin_2021}, and Bose--Einstein condensates~\cite{doi:10.1016/j.xinn.2023.100546}.

Despite these successes, the efficiency of HiSD is strongly governed by the spectral condition number $\kappa$ of the Hessian at the saddle point~\cite{Luo2022Convergence}. Specifically, based on the rigorous error and convergence analysis of its Euler discretization~\cite{Zhang2022error}, the discrete scheme converges linearly with contraction factor $q = (\kappa-1)/(\kappa+1)$, implying an iteration complexity of $O(\kappa\log(1/\epsilon))$.
For problems involving stiff potentials, fine spatial discretizations, or large spectral separation, the Hessian can be severely ill-conditioned. As a result, the admissible step size becomes very small and the convergence can be slow. 
While recent momentum-based acceleration techniques~\cite{luo2025accelerated} and advanced discretization schemes, such as predictor-corrector methods \cite{LI2025108852}, have been proposed, they do not change the unfavorable spectral properties of the underlying operator.

In classical optimization, preconditioning is a standard strategy for accelerating iterative methods by transforming the original geometry into one with a more favorable spectral structure~\cite{Nesterov2018}. In the context of saddle point search, preconditioning has been successfully incorporated into several index-1 methods. For example, Gould et al.~\cite{Gould2016dimertype} introduced preconditioning into the dimer method and established local linear convergence, while Packwood et al.~\cite{Packwood2016Universal} developed universal preconditioners for transition-path computations in molecular systems. However, for \emph{high-index} saddle dynamics, a systematic preconditioning framework together with rigorous stability and convergence analysis is still lacking.

To address this difficulty, we propose the \emph{preconditioned high-index saddle dynamics} (p-HiSD). We introduce a Riemannian metric on $\mathbb{R}^n$ induced by a symmetric positive definite (SPD) preconditioner $M$ and reformulate the saddle dynamics within this metric. The main contributions of this paper are summarized as follows:

\begin{enumerate}[label=(\roman*)]
\item We derive the p-HiSD equations by generalizing the standard reflection and deflation operators to their $M$-orthogonal counterparts. The dynamics are driven by the $M$-gradient $\nabla_M E = M^{-1}\nabla E$ and the generalized eigenpairs satisfying $Hv = \lambda Mv$. We prove that the critical points and their Morse indices are invariant under this change of metric (Proposition~\ref{prop:invariance}).

\item We establish local exponential stability of p-HiSD (Theorem~\ref{thm:stability}) and prove that the discrete scheme converges linearly with rate $(\kappa_M-1)/(\kappa_M+1)$ (Theorem~\ref{thm:convergence}), where $\kappa_M$ denotes the preconditioned condition number. The analysis also accounts for errors arising from inexact eigenspace computations. This shows that, when $M$ captures the dominant Hessian anisotropy, the effective conditioning of the iteration can be substantially improved.

\item We develop a suite of practical preconditioners (e.g., spectral, subspace-inertial, block Jacobi, incomplete Cholesky) and validate them across stiff model problems and PDE discretizations. Numerical results demonstrate that p-HiSD permits stable step sizes up to $10^3$--$10^4$ times larger than those of the standard method, significantly reduces iteration counts, and remains robust in stiff regimes where standard HiSD may fail. 
\end{enumerate}

The remainder of this paper is organized as follows. Section~\ref{sec:preliminaries} reviews the standard HiSD formulation. Section~\ref{sec:p-hisd} introduces the Riemannian geometry induced by $M$ and derives the p-HiSD equations. Section~\ref{sec:stability} establishes local exponential stability. Section~\ref{sec:discrete} provides the discrete convergence analysis. Section~\ref{sec:preconditioners} discusses representative preconditioner designs and a heuristic selection rule used in our implementation. Numerical experiments are reported in Section~\ref{sec:numerical}, and concluding remarks are offered in Section~\ref{sec:conclusion}.

\section{Preliminaries}
\label{sec:preliminaries}
High-index saddle dynamics (HiSD) provides a dynamical-systems approach for computing saddle points of prescribed Morse index~\cite{Yin2019HiSD}. We begin by recalling the original formulation and the convergence result most relevant to the present work.

Let $E\in C^2(\mathbb{R}^n,\mathbb{R})$ with gradient $\nabla E(x)$ and Hessian $H(x)=\nabla^2 E(x)$. A critical point $x^\ast$ (with $\nabla E(x^\ast)=0$) is said to have \emph{Morse index} $k$ if $H(x^\ast)$ has exactly $k$ negative eigenvalues, counted with multiplicity. The goal of HiSD is to compute such an index-$k$ saddle by evolving the state variable $x$ together with a $k$-frame that tracks the unstable eigenspace of the local Hessian.

Specifically, HiSD consists of the coupled system
\begin{equation}
\label{eq:hisd}
\begin{aligned}
\dot{x} &= -\eta \mathcal{R}_V \nabla E(x), \\
\dot{v}_i &= -\tau \mathcal{P}_i H(x) v_i, \quad i = 1, \ldots, k,
\end{aligned}
\end{equation}
where $\eta,\tau>0$, the frame $V=[v_1,\dots,v_k]$ satisfies $v_i^\top v_j=\delta_{ij}$, and 
\begin{align*}
\mathcal{R}_V &= I - 2\sum_{i=1}^k v_i v_i^\top,\quad 
\mathcal{P}_i = I - v_i v_i^\top - 2\sum_{j<i} v_j v_j^\top. \label{eq:deflation}
\end{align*}
The operator $\mathcal{R}_V$ reverses the component of $\nabla E(x)$ along these directions and leaves the complementary component unchanged, so the state variable ascends along the approximate unstable subspace and descends in the remaining directions. 
Meanwhile, the frame equation updates the vectors $v_i$ to track the negative eigendirections of the Hessian, while the sequential deflation in $\mathcal{P}_i$ maintains orthogonality and ordering. 
Note that HiSD is not a gradient flow in the usual Riemannian sense, because $\mathcal{R}_V$ is indefinite, with eigenvalue $-1$ on $\operatorname{span}(V)$ and $+1$ on $\operatorname{span}(V)^\perp$.

The local convergence theory of HiSD shows that this coupled dynamics stabilizes the desired saddle point together with the associated unstable frame.

\begin{theorem}[Local stability~\cite{Yin2019HiSD}]
\label{thm:hisd-stability-original}
Let $x^*$ be a nondegenerate index-$k$ saddle with Hessian eigenvalues $\lambda_1 \le \cdots \le \lambda_k < 0 < \lambda_{k+1} \le \cdots \le \lambda_n$, all distinct. Let $\{u_i\}_{i=1}^k$ be eigenvectors corresponding to the negative eigenvalues. Then $(x^*, u_1, \ldots, u_k)$ is locally exponentially stable for~\eqref{eq:hisd}.
\end{theorem}

A practical implementation discretizes the state update explicitly and recomputes the unstable frame from the Hessian at the new iterate:
\begin{equation*}
\label{eq:discrete-hisd}
x_{m+1} = x_m - \eta \mathcal{R}_{V_m}\nabla E(x_m),
\qquad
V_{m+1} = \mathrm{Eig}_k(H(x_{m+1})),
\end{equation*}
where $V_m=[v_{m,1},\ldots,v_{m,k}]$ and $\mathrm{Eig}_k$ returns $k$ orthonormal eigenvectors associated with the algebraically smallest eigenvalues. The following result shows that the local convergence rate of this discrete scheme is controlled by the spectral condition number of the Hessian.

\begin{theorem}[Discrete convergence~\cite{Luo2022Convergence}]
\label{thm:discrete-hisd-original}
Assume in a neighborhood of $x^*$:
\begin{enumerate}
\item $H$ is Lipschitz: $\|H(x)-H(y)\|_2 \le K\|x-y\|$;
\item the eigenvalues of $H(x^*)$ satisfy $\mu \le |\lambda_i| \le L$.
\end{enumerate}
Let $\kappa=L/\mu$. With $\eta=2/(L+\mu)$ and $\|x_0-x^*\|$ sufficiently small,
\begin{equation}
\label{eq:hisd-rate}
\|x_{m+1}-x^*\| \le \frac{\kappa-1}{\kappa+1}\|x_m-x^*\| + O(\|x_m-x^*\|^2).
\end{equation}
\end{theorem}

The estimate~\eqref{eq:hisd-rate} makes the main numerical bottleneck transparent: the contraction factor $q=(\kappa-1)/(\kappa+1)$ deteriorates as $\kappa$ becomes large.
For stiff or highly anisotropic problems, the Euclidean geometry underlying both the reflected state update and the eigenspace computation therefore leads to small admissible step sizes and slow convergence. This observation motivates the central idea of the present work: rather than changing the saddle-search mechanism itself, we replace the Euclidean geometry by one induced by a symmetric positive definite preconditioner, with the goal of improving conditioning while preserving the target saddle structure.

\section{Preconditioned High-Index Saddle Dynamics}
\label{sec:p-hisd}
The preceding discussion suggests that the efficiency of HiSD is limited not by its saddle-search mechanism itself, but by the Euclidean metric underlying the reflected state update and frame dynamics. For ill-conditioned problems, we therefore replace this geometry by one induced by an SPD preconditioner, while preserving the target saddle structure.

\subsection{A metric reformulation of HiSD}
We now replace the three geometric ingredients of HiSD—the gradient, the unstable subspace, and the reflection/deflation operators—by their $M$-metric counterparts.

Let $M\in\mathbb{R}^{n\times n}$ be a symmetric positive definite (SPD) matrix. We equip $\mathbb{R}^n$ with the $M$-inner product
\begin{equation*}
\label{eq:M-inner-product}
\langle u,v\rangle_M := u^\top Mv,
\qquad
\|u\|_M := \sqrt{u^\top M u}.
\end{equation*}
Whenever needed, the induced operator norm is defined by
\begin{equation*}
\label{def:M-norm}
\|A\|_M := \sup_{x\neq 0}\frac{\|Ax\|_M}{\|x\|_M}
= \|M^{1/2} A M^{-1/2}\|_2,
\qquad A\in\mathbb{R}^{n\times n}.
\end{equation*}
Under this metric, the gradient of $E$ becomes
\begin{equation}
\label{eq:M-gradient}
\nabla_M E(x) = M^{-1} \nabla E(x),
\end{equation}
and the unstable directions are described by the generalized eigenvalue problem
\begin{equation}
\label{eq:gen-eig}
H(x)v=\lambda Mv.
\end{equation}

Since $M$ is SPD and $H(x)$ is symmetric, the generalized eigenvalues are real, and the eigenvectors may be chosen to be $M$-orthonormal. Thus, if $V=[v_1,\dots,v_k]$ satisfies $V^\top M V=I_k$, then $\operatorname{span}(V)$ serves as the unstable subspace in the preconditioned geometry, and the associated $M$-orthogonal projection is
\begin{equation*}
\mathcal{P}_V^M = VV^\top M.
\end{equation*}

With these ingredients, the Euclidean reflection and deflation operators in HiSD admit natural $M$-orthogonal analogues:
\begin{align*}
\mathcal{R}_V^M &= I - 2\sum_{i=1}^k v_i v_i^\top M, 
\qquad \text{($M$-orthogonal reflection)}, \label{eq:M-reflection}\\
\mathcal{P}_i^M &= I - v_i v_i^\top M - 2\sum_{j<i} v_j v_j^\top M,
\qquad \text{(deflated $M$-projection)}.
\end{align*}
The first operator reverses the $M$-orthogonal component along $\operatorname{span}(V)$, while the second enforces the sequential orthogonality constraints needed in the frame dynamics. Replacing the Euclidean gradient and projection operators in~\eqref{eq:hisd} by their $M$-metric counterparts leads to the preconditioned system

\begin{equation}
\label{eq:p-hisd}
\begin{aligned}
\dot{x} &= -\eta \mathcal{R}_V^M M^{-1}\nabla E(x), \\
\dot{v}_i &= -\tau \mathcal{P}_i^M M^{-1}H(x)v_i, \qquad i=1,\ldots,k,
\end{aligned}
\end{equation}
where the frame vectors satisfy $v_i^\top M v_j=\delta_{ij}$. When $M=I$, \eqref{eq:p-hisd} reduces to the original HiSD dynamics~\eqref{eq:hisd}.

\subsection{Invariance of the target saddle structure}

A basic requirement for any preconditioned reformulation is that it should not alter the target critical points or their Morse indices. The next result shows that p-HiSD satisfies exactly this property.

\begin{proposition}
\label{prop:invariance}
For any SPD preconditioner $M$:
\begin{enumerate}[label=(\alph*)]
\item $\nabla_M E(x^*) = 0$ if and only if $\nabla E(x^*) = 0$.\label{grad_zero}
\item The number of negative generalized eigenvalues of $(H(x^*), M)$ is equal to the number of negative eigenvalues of $H(x^*)$.\label{precond_index}
\end{enumerate}
\end{proposition}

\begin{proof}
Claim~\ref{grad_zero} follows directly from~\eqref{eq:M-gradient} and the invertibility of $M$.

For Claim~\ref{precond_index}, consider the symmetric matrix $\tilde{H} = M^{-1/2} H(x^*) M^{-1/2}$. By Sylvester's law of inertia, the congruent matrices $H(x^*)$ and $\widetilde H$ have the same inertia, i.e., the same numbers of positive, negative, and zero eigenvalues.

On the other hand, $\widetilde H$ is similar to $M^{-1}H(x^*)$ through
\begin{equation*}
\widetilde H = M^{1/2}(M^{-1}H(x^*))M^{-1/2},
\end{equation*}
so they have the same spectrum. Moreover, since
\begin{equation*}
\det(H(x^*)-\lambda M)=0
\quad\Longleftrightarrow\quad
\det(M^{-1}H(x^*)-\lambda I)=0,
\end{equation*}
the eigenvalues of $M^{-1}H(x^*)$ are precisely the generalized eigenvalues of $(H(x^*),M)$. Therefore, the number of negative generalized eigenvalues of $(H(x^*),M)$ coincides with the number of negative eigenvalues of $H(x^*)$.
\end{proof}

Proposition~\ref{prop:invariance} shows that preconditioning changes the geometry of the dynamics without changing the identity of the target saddles. This makes p-HiSD a legitimate acceleration mechanism for HiSD: it modifies conditioning and stability properties while preserving the critical points and their Morse indices.

\section{Local Exponential Stability}
\label{sec:stability}
In this section, we analyze the local behavior of p-HiSD near equilibrium. We first characterize the equilibria of~\eqref{eq:p-hisd}, and then determine which of them are locally exponentially stable.

We begin by identifying the stationary states of p-HiSD.

\begin{proposition}[Equilibrium characterization]
\label{prop:equilibrium}
Let $\{u_i\}_{i=1}^k$ be a set of $M$-orthonormal vectors (i.e., $u_i^\top M u_j = \delta_{ij}$). A state $Z^* = (x^*, u_1, \ldots, u_k)$ is an equilibrium of the p-HiSD system~\eqref{eq:p-hisd} if and only if $x^*$ is a critical point of $E$ and, for each $i=1,\ldots,k$, $u_i$ is a generalized eigenvector of $(H(x^*),M)$ corresponding to some generalized eigenvalue $\lambda_i$.
\end{proposition}

\begin{proof}
We prove the two implications separately.


\textit{Sufficiency:} Suppose $\nabla E(x^*)=0$. Evaluating the $x$-component of~\eqref{eq:p-hisd} at $Z^*$ yields
$\dot x(Z^*)=-\eta\mathcal{R}_V^M M^{-1}\nabla E(x^*)=0$.
For each $i$, the relation $H(x^*)u_i=\lambda_i M u_i$ implies $M^{-1}H(x^*)u_i=\lambda_i u_i$. By $M$-orthonormality, the projection term vanishes:
\begin{align*}
\mathcal{P}_i^M u_i = u_i-u_i(u_i^\top M u_i)-2\sum_{j<i}u_j(u_j^\top M u_i)=0.
\end{align*}

Consequently, the $v_i$-component of~\eqref{eq:p-hisd} is zero at $Z^*$:
$$\dot v_i(Z^*)=-\tau\mathcal{P}_i^M M^{-1}H(x^*)u_i=0.$$

\textit{Necessity:} Conversely, assume that $Z^*=(x^*,u_1,\dots,u_k)$ is an equilibrium. Since $\mathcal{R}_V^M$ is invertible, the condition $\dot x(Z^*)=0$ implies $\nabla E(x^*)=0$.

It remains to show that each $u_i$ is a generalized eigenvector. We argue by induction on $i$. For $i=1$, the equilibrium condition $\dot v_1(Z^*)=0$ gives
\begin{align*}
(I-u_1u_1^\top M)M^{-1}H(x^*)u_1=0,
\end{align*}
hence $H(x^*)u_1$ is parallel to $Mu_1$, with
$$H(x^*)u_1=\lambda_1Mu_1,\quad\lambda_1=u_1^\top H(x^*) u_1\in\mathbb{R}.$$

Now assume that $H(x^*)u_j=\lambda_jMu_j$ holds for all $j<i$. Since both $H(x^*)$ and $M$ are symmetric, taking the transpose yields $u_j^\top H(x^*)=\lambda_j u_j^\top M$. The $M$-orthonormality then implies $u_j^\top H(x^*)u_i=\lambda_j u_j^\top M u_i=0$ for all $j<i$.
Substituting this into the equilibrium condition $\dot v_i(Z^*)=0$ and using the definition of $\mathcal{P}_i^M$, we obtain
\begin{equation*}
\begin{aligned}
0 &= \mathcal{P}_i^M M^{-1}H(x^*)u_i = M^{-1}H(x^*)u_i-(u_i^\top H(x^*)u_i)u_i-2\sum_{j<i}u_j\,(u_j^\top H(x^*)u_i) \\
&= M^{-1}H(x^*)u_i-(u_i^\top H(x^*)u_i)u_i.
\end{aligned}
\end{equation*}
Thus, letting $\lambda_i:=u_i^\top H(x^*)u_i$ yields the generalized eigenvalue equation $H(x^*)u_i=\lambda_iMu_i$.
\end{proof}

After identifying all equilibria of p-HiSD in Proposition~\ref{prop:equilibrium}, we now determine which of them are attracting. The next theorem shows that local exponential stability is obtained precisely when the frame is aligned with the generalized eigendirections corresponding to the $k$ negative generalized eigenvalues.

\begin{theorem}[Local exponential stability]
\label{thm:stability}
Let $\mathcal{Z}^* = (x^*, u_1, \ldots, u_k)$ be an equilibrium point characterized by Proposition~\ref{prop:equilibrium}. Then $\mathcal{Z}^*$ is locally exponentially stable for the system~\eqref{eq:p-hisd} if and only if:
\begin{enumerate}
\item $x^*$ is a nondegenerate index-$k$ saddle of $E$ (i.e., $(H(x^*),M)$ has exactly $k$ negative generalized eigenvalues);\label{local1}
\item The vectors $u_1,\ldots,u_k$ are the generalized eigenvectors associated with these $k$ negative generalized eigenvalues, ordered such that
$\lambda_1 < \cdots < \lambda_k < 0$.\label{local2}
\end{enumerate}
\end{theorem}

\begin{proof}
We extend the linearization analysis in~\cite{Yin2019HiSD} to the Riemannian setting induced by $M$. Writing the system~\eqref{eq:p-hisd} as $\dot Z = F(Z)$ for the augmented state $Z=(x,v_1,\ldots,v_k)$. To establish local exponential stability, we analyze the spectrum of the Jacobian $\mathcal{J} := DF(Z^*)$ at equilibrium. Partitioning $\mathcal{J}$ according to the splitting of $Z$ yields
\begin{equation*}
\mathcal{J} = \begin{pmatrix}
\mathcal{J}_{xx} & \mathcal{J}_{xv} \\
\mathcal{J}_{vx} & \mathcal{J}_{vv}
\end{pmatrix}.
\end{equation*}
At the equilibrium $Z^*$, a first-order perturbation in the frame $(v_1,\ldots,v_k)$ affects $\dot{x}$ solely through the reflector $\mathcal{R}_V^M$. Since $\nabla E(x^*) = 0$, this dependence vanishes, and thus $\mathcal{J}_{xv} = 0$. Consequently, $\mathcal{J}$ is block lower triangular, and its spectrum is given by $\sigma(\mathcal{J}) = \sigma(\mathcal{J}_{xx}) \cup \sigma(\mathcal{J}_{vv})$. We compute these spectra using the generalized eigenbasis of $(H(x^*), M)$.
Let $\{(\lambda_j, u_j)\}_{j=1}^n$ be a complete set of generalized eigenpairs, where $u_1, \dots, u_k$ correspond to the equilibrium state and $u_{k+1}, \dots, u_n$ complete an $M$-orthonormal basis.

First, we analyze $\mathcal{J}_{xx}$. Linearizing $\dot x=-\eta\mathcal{R}_V^M M^{-1}\nabla E(x)$ yields
\begin{equation*}
\mathcal{J}_{xx}=-\eta\mathcal{R}_V^M M^{-1}H(x^*), \quad \text{where} \quad \mathcal{R}_V^M=I-2\sum_{i=1}^k u_i u_i^\top M .
\end{equation*}
The reflection operator acts as $\mathcal{R}_V^M u_j = -u_j$ on the subspace spanned by $\{u_1, \dots, u_k\}$ and as $\mathcal{R}_V^M u_j = u_j$ on its orthogonal complement, and thus
\begin{equation}
\label{eig:J_xx}
\mathcal{J}_{xx}u_j=-\eta\mathcal{R}_V^M(\lambda_j u_j)=
\begin{cases}
\eta\lambda_j u_j, & j\le k,\\
-\eta\lambda_j u_j, & j>k.
\end{cases}
\end{equation}
Next, we consider $\mathcal{J}_{vv}$. Due to the deflation structure, the dynamics of $v_i$ depends only on $\{v_1,\dots,v_i\}$, rendering $\mathcal{J}_{vv}$ block lower triangular. Hence, $\sigma(\mathcal{J}_{vv})=\bigcup_{i=1}^k\sigma(\mathcal{J}_{v_i v_i})$. Adapting the linearization from~\cite{Yin2019HiSD} to the $M$-inner product, the diagonal block at $Z^*$ is
\begin{equation*}
\mathcal{J}_{v_i v_i}=\tau\Big(\lambda_i I + 2\sum_{j=1}^i \lambda_j u_j u_j^\top M - M^{-1}H(x^*)\Big).
\end{equation*}
Applying this operator to the basis vector $u_\ell$ yields
\begin{equation}
\label{eig:J_ii}
\mathcal{J}_{v_i v_i}u_\ell=
\begin{cases}
\tau(\lambda_i+\lambda_\ell)u_\ell, & \ell\le i,\\
\tau(\lambda_i-\lambda_\ell)u_\ell, & \ell> i.
\end{cases}
\end{equation}
The stability of $\mathcal{Z}^*$ is determined by the sign of the real parts of the eigenvalues in~\eqref{eig:J_xx} and~\eqref{eig:J_ii}. We now prove the necessity and sufficiency of conditions \ref{local1} and \ref{local2}.

\textit{Sufficiency:} Suppose conditions \ref{local1} and \ref{local2} hold. Then $\lambda_1<\cdots<\lambda_k<0$. Since $x^*$ is a nondegenerate index-$k$ saddle, we also have $\lambda_\ell>0$ for all $\ell>k$.
From~\eqref{eig:J_xx}, we find that for $j \le k$, the eigenvalue is $\eta\lambda_j < 0$, and for $j > k$, it is $-\eta\lambda_j < 0$. Thus, $\sigma(\mathcal{J}_{xx}) \subset \mathbb{R}^-$.
From~\eqref{eig:J_ii}, consider the term $\tau(\lambda_i \pm \lambda_\ell)$. If $\ell \le i$, then $\tau(\lambda_i + \lambda_\ell) < 0$ (sum of two negatives). If $\ell > i$, we distinguish two cases:
(a) $i < \ell \le k$: here $\lambda_i < \lambda_\ell$ by ordering (2), so $\tau(\lambda_i - \lambda_\ell) < 0$;
(b) $\ell > k$: here $\lambda_\ell > 0 > \lambda_i$, so $\tau(\lambda_i - \lambda_\ell) < 0$.
In all cases, the eigenvalues are negative. Thus, $\sigma(\mathcal{J}) \subset \mathbb{R}^-$, implying local exponential stability.

\textit{Necessity:} Conversely, assume $\mathcal{Z}^*$ is locally exponentially stable. The condition $\sigma(\mathcal{J}_{xx}) \subset \mathbb{R}^-$ combined with~\eqref{eig:J_xx} implies $\eta\lambda_j < 0$ for $j \le k$ and $-\eta\lambda_j < 0$ (i.e., $\lambda_j > 0$) for $j > k$. This signifies that $x^*$ is exactly an index-$k$ saddle. Furthermore, the condition $\sigma(\mathcal{J}_{vv}) \subset \mathbb{R}^-$ requires that for any $i \in \{1,\dots,k\}$ and $\ell > i$, the eigenvalue $\tau(\lambda_i - \lambda_\ell)$ must be negative. Since $\tau > 0$, this enforces the strict ordering $\lambda_1 < \dots < \lambda_k$, implying that $u_1,\dots,u_k$ must be the eigenvectors corresponding to the $k$ negative eigenvalues sorted by magnitude.
\end{proof}

The theorem shows that p-HiSD has exactly the expected stable equilibria: the state variable converges to an index-$k$ saddle, while the frame aligns with the unstable generalized eigenspace.

\begin{remark}
The same Jacobian argument shows that replacing the constant matrix $M$ by a smooth, state-dependent SPD metric $M(x)$ does not change the linearized spectrum at equilibrium. The additional terms arising from differentiating $M(x)^{-1}$ and $\mathcal{R}_V^{M(x)}$ are multiplied by $\nabla E(x^*)$ and therefore vanish at the critical point, while the block lower triangular structure is preserved.
\end{remark}

\section{Discrete Convergence Analysis}
\label{sec:discrete}

We now turn to the discrete p-HiSD iteration and study its local convergence near a nondegenerate index-$k$ saddle.

We adopt the discrete preconditioned HiSD scheme summarized in Algorithm~\ref{alg:p-hisd}. At each outer step, the state variable is updated by a reflected preconditioned descent direction, while the unstable frame is approximated by a small number of inner iterations applied to the generalized eigenvalue problem.

\begin{algorithm}[htbp]
\caption{Discrete Preconditioned HiSD (p-HiSD)}
\label{alg:p-hisd}
\begin{algorithmic}[1]
\Require Initial guess $x_0$, preconditioner $M$, index $k$, step size $\eta$, eigenvector step size $\tau$, inner iterations $J$, tolerance $\varepsilon$
\Ensure Approximate index-$k$ saddle point $x^*$
\State $m \gets 0$; initialize $V_0 = [v_1, \ldots, v_k]$, $M$-orthonormalize
\Repeat
    \State $g_m \gets \nabla E(x_m)$
    \State $H_m \gets \nabla^2 E(x_m)$
    \For{$j = 1, \ldots, J$} \Comment{Inner iterations: iterative generalized eigensolve}
        \For{$i = 1, \ldots, k$}
            \State $v_i \gets v_i - \tau \, \mathcal{P}_i^M \, M^{-1} H_m \, v_i$ \Comment{Rayleigh quotient minimization}
        \EndFor
        \State $M$-orthonormalize $V_m = [v_1, \ldots, v_k]$ \Comment{Enforce $V_m^\top M V_m = I_k$}
    \EndFor
    \State  $ \tilde{g}_m \gets M^{-1}g_m$
    \State $d_m \gets -\tilde{g}_m + 2 V_m (V_m^\top g_m)$ \Comment{$O(nk)$ operations}
    \State $x_{m+1} \gets x_m + \eta \, d_m$
    \State $m \gets m + 1$
\Until{$\|g_m\| < \varepsilon$}
\State \Return $x_m$
\end{algorithmic}
\end{algorithm}

The inner loop is a finite-step discretization of the continuous frame dynamics $\dot{v}_i = - \tau \mathcal{P}_i^M M^{-1} H v_i$ from~\eqref{eq:p-hisd}, and can be interpreted as Riemannian gradient descent for the generalized Rayleigh quotient $\frac{v^\top H v}{v^\top M v}$ on the generalized Stiefel manifold
. In the ideal limit $J\to\infty$, it recovers the generalized eigenspace associated with the $k$ smallest eigenvalues. In practice, however, only a small number of inner iterations is needed because the outer iterates move gradually and the previous frame provides an effective warm start. As a result, the unstable eigenspace is tracked only approximately, and this approximation error must be incorporated into the convergence analysis.

To quantify this effect, let $V_m$ denote the exact $M$-orthonormal basis of the unstable generalized eigenspace of $(H(x_m),M)$, and let $\widehat V_m$ denote the basis actually used by the algorithm. The following assumptions encode the local regularity of the Hessian and the accuracy of this inexact eigenspace approximation.

\begin{assumption}[Regularity in $M$-metric]
\label{ass:regularity}
In a neighborhood $U(x^*, \delta)$ of the saddle point $x^*$, we assume:
\begin{enumerate}[label=(\alph*)]
\item Hessian Lipschitz continuity: There exists $K > 0$ such that
\begin{equation*}
\|M^{-1}(H(x) - H(y))\|_M \le K\|x - y\|_M \quad \forall x, y \in U(x^*, \delta).
\end{equation*}

\item Spectral bounds: The point $x^*$ is a nondegenerate index-$k$ saddle. The generalized eigenvalues of $(H(x^*),M)$ satisfy
\begin{equation*}
-L\le \lambda_1 \le \dots \le \lambda_k \le -\mu < 0 < \mu \le \lambda_{k+1} \le \dots \le \lambda_n \le L,
\end{equation*}
for some constants $0<\mu\le L<\infty$.

\item Inexact projector: There exists $\theta\ge 0$ such that, for all $m$ where $x_m\in U(x^*,\delta)$,
\begin{equation*}\label{eq:proj_err}
\bigl\| V_m V_m^\top M - \widehat{V}_m \widehat{V}_m^\top M \bigr\|_M \le \theta.
\end{equation*}
\end{enumerate}
\end{assumption}

Under Assumption~\ref{ass:regularity}, the discrete p-HiSD iteration inherits a local linear convergence rate governed by the \emph{preconditioned condition number} $\kappa_M := \frac{L}{\mu}$, up to a perturbation induced by the inexact eigenspace computation.

\begin{theorem}[Local linear convergence]
\label{thm:convergence}
Under Assumption~\ref{ass:regularity}, suppose $\|x_0 - x^*\|_M$ is sufficiently small.
If the eigenvector computation error satisfies $\theta < \frac{1}{2\kappa_M}$,
then Algorithm~\ref{alg:p-hisd} with step size $\eta = \frac{2}{L + \mu}$ satisfies
\begin{equation*}
\label{eq:main-convergence}
\|x_{m+1} - x^*\|_M \le \left( \frac{\kappa_M - 1}{\kappa_M + 1} + \frac{4\kappa_M}{\kappa_M + 1}\theta \right) \|x_m - x^*\|_M + C\|x_m - x^*\|_M^2
\end{equation*}
for some constant $C > 0$. Consequently, the iteration converges linearly to $x^*$.
\end{theorem}

\begin{proof}
Let $e_m = x_m - x^*$ denote the error at step $m$. We define the exact and computed reflection operators as $\mathcal{R}_m = I - 2V_m V_m^\top M$ and $\widehat{\mathcal{R}}_m = I - 2\widehat{V}_m \widehat{V}_m^\top M$, respectively.

Expanding the gradient near $x^*$, we write $\nabla E(x_m) = H(x^*) e_m + r_m$, where the remainder $r_m = \int_{0}^{1} \left (H(x^* + t e_m) - H(x^* ) \right )  e_m \, dt$ satisfies $\|M^{-1}r_m\|_M \le \frac{K}{2}\|e_m\|_M^2$ indicated by the Lipschitz condition in Assumption~\ref{ass:regularity}(a). 
Therefore, the update step $x_{m+1} = x_m - \eta \widehat{\mathcal{R}}_m M^{-1} \nabla E(x_m)$ yields the error recursion
\begin{equation*}
\label{eq:error-recursion-split}
\begin{aligned}
e_{m+1} &= \underbrace{\left( e_m - \eta \mathcal{R}_m M^{-1} (H(x^*)e_m + r_m) \right)}_{\mathrm{Term\ I}} + \underbrace{\eta (\mathcal{R}_m - \widehat{\mathcal{R}}_m) M^{-1} \nabla E(x_m)}_{\mathrm{Term\ II}}.
\end{aligned}
\end{equation*}

We first analyze $\mathrm{Term\ I}$ by isolating the linearized operator at equilibrium. Let $T^* := \mathcal{R}^* M^{-1} H(x^*)$, where $\mathcal{R}^*$ denotes the reflection operator constructed from the exact eigenvectors at $x^*$. We decompose this term as
\begin{equation*}
\mathrm{Term\ I} = (I - \eta T^*) e_m + \eta (\mathcal{R}^* - \mathcal{R}_m) M^{-1} H(x^*) e_m - \eta \mathcal{R}_m M^{-1} r_m.
\end{equation*}
Taking norms and observing that $\|\mathcal{R}_m\|_M = 1$ and $\|M^{-1}H(x^*)\|_M \le L$, we obtain
\begin{equation}
\label{eq:term1-bound-raw}
\|\mathrm{Term\ I}\|_M \le \|I - \eta T^*\|_M \|e_m\|_M + \eta L \|\mathcal{R}^* - \mathcal{R}_m\|_M \|e_m\|_M + \frac{\eta K}{2}\|e_m\|_M^2.
\end{equation}
We then estimate the components of~\eqref{eq:term1-bound-raw} as follows. 

First, we bound the operator norm $\|I - \eta T^*\|_M$. Let $\{(\lambda_j, u_j)\}_{j=1}^n$ be the $M$-orthonormal generalized eigenpairs of $(H(x^*), M)$. Since $M^{-1}H(x^*) u_j = \lambda_j u_j$, the exact reflection operator $\mathcal{R}^*$ flips the sign of $u_j$ for $j \le k$ (where $\lambda_j < 0$) and preserves it for $j > k$ (where $\lambda_j > 0$). Thus, the linearized operator $T^* := \mathcal{R}^* M^{-1} H(x^*)$ acts on the basis as
\begin{equation*}
T^* u_j = \mathcal{R}^* (\lambda_j u_j) = |\lambda_j| u_j, \quad j = 1, \dots, n.
\end{equation*}

To bound the induced $M$-operator norm of $I - \eta T^*$, we expand an arbitrary vector $x \in \mathbb{R}^n$ in this $M$-orthonormal basis as $x = \sum_{j=1}^n c_j u_j$. Applying the operator yields $(I - \eta T^*)x = \sum_{j=1}^n c_j (1 - \eta |\lambda_j|) u_j$. By the definition of the $M$-norm and the $M$-orthonormality of $\{u_j\}$, we have
\begin{equation*}
\|(I - \eta T^*)x\|_M^2 = \sum_{j=1}^n c_j^2 (1 - \eta |\lambda_j|)^2 \le \max_{1 \le j \le n} (1 - \eta |\lambda_j|)^2 \sum_{j=1}^n c_j^2 = \max_{1 \le j \le n} (1 - \eta |\lambda_j|)^2 \|x\|_M^2.
\end{equation*}
Taking the supremum over all $x \neq 0$ and substituting the step size $\eta = \frac{2}{L+\mu}$, we obtain
\begin{equation*}
\|I - \eta T^*\|_M \le \max_{\nu \in [\mu, L]} |1 - \eta \nu| \le \frac{L-\mu}{L+\mu} = \frac{\kappa_M - 1}{\kappa_M + 1}.
\end{equation*}

Second, we bound the subspace perturbation $\|\mathcal{R}_m - \mathcal{R}^*\|_M$.
Since $\mathcal{R}_m$ and $\mathcal{R}^*$ are reflections across the spaces spanned by $V_m$ and $U_*$ respectively, where $U_*$ consists of the generalized eigenvectors of $(H(x^*),M)$ corresponding to the unstable directions. Their difference in the $M$-norm is controlled by the canonical angles between these subspaces:
\begin{equation*}
\|\mathcal{R}_m - \mathcal{R}^*\|_M = 2 \|\sin\Theta(V_m, U_*)\|_M.
\end{equation*}

To apply standard perturbation theory, we consider the transformed Hessian $\tilde{H}(x) := M^{-1/2}H(x)M^{-1/2}$. The generalized eigenpairs of $(H(x), M)$ are in one-to-one correspondence with the standard eigenpairs of $\tilde{H}(x)$ via the mapping $v \mapsto M^{1/2}v$. Consequently, the angle between the subspaces in the $M$-metric is identical to the Euclidean angle between the transformed subspaces $\tilde{V}_m = M^{1/2}V_m$ and $\tilde{U}_* = M^{1/2}U_*$.
The matrix $\tilde{H}(x^*)$ has a spectral gap $\delta_{\text{gap}} \ge 2\mu$ separating its negative and positive eigenvalues. Applying the Davis--Kahan $\sin\Theta$ theorem~\cite{davis1970rotation} to the perturbation $\tilde{H}(x_m) - \tilde{H}(x^*)$, we obtain
\begin{equation*}
\|\sin\Theta(\tilde{V}_m, \tilde{U}_*)\|_2 \le \frac{\|\tilde{H}(x_m) - \tilde{H}(x^*)\|_2}{\delta_{\text{gap}}} \le \frac{\|M^{-1}(H(x_m) - H(x^*))\|_M}{2\mu}.
\end{equation*}
Using the Lipschitz continuity from Assumption~\ref{ass:regularity}(a), this simplifies to
\begin{equation*}
\|\mathcal{R}_m - \mathcal{R}^*\|_M \le 2 \left( \frac{K\|e_m\|_M}{2\mu} \right) = \frac{K}{\mu}\|e_m\|_M.
\end{equation*}
Substituting these estimates back into~\eqref{eq:term1-bound-raw}, we obtain
\begin{equation}
\label{eq:term1-final}
\|\mathrm{Term\ I}\|_M \le \frac{\kappa_M - 1}{\kappa_M + 1} \|e_m\|_M + \eta \left(\frac{KL}{\mu} + \frac{K}{2}\right) \|e_m\|_M^2.
\end{equation}

Next, we turn to $\mathrm{Term\ II}$, which accounts for the error arising from the inexact projection. Assumption~\ref{ass:regularity}(c) guarantees that $\|\mathcal{R}_m - \widehat{\mathcal{R}}_m\|_M \le 2\theta$. Furthermore, using the spectral bound $L$ and the remainder estimate derived earlier, the gradient term satisfies $\|M^{-1}\nabla E(x_m)\|_M \le L\|e_m\|_M + \frac{K}{2}\|e_m\|_M^2$. Combining these bounds yields
\begin{equation}
\label{eq:term2-final}
\|\mathrm{Term\ II}\|_M \le \eta (2\theta) \left( L\|e_m\|_M + \frac{K}{2}\|e_m\|_M^2 \right) = 2\eta L \theta \|e_m\|_M + \eta\theta K \|e_m\|_M^2.
\end{equation}

Combining~\eqref{eq:term1-final} and~\eqref{eq:term2-final} with $\eta=\frac{2}{L+\mu}$ gives the error recursion
\begin{equation}
\label{ineq:convergence}
\|e_{m+1}\|_M \le \left( \frac{\kappa_M - 1}{\kappa_M + 1} + \frac{4\kappa_M}{\kappa_M + 1}\theta \right) \|e_m\|_M + C \|e_m\|_M^2,
\end{equation}
where $C>0$ collects all quadratic constants. Let $q$ denote the linear rate factor:
\begin{equation}
q := \frac{\kappa_M - 1}{\kappa_M + 1} + \frac{4\kappa_M}{\kappa_M + 1}\theta.
\end{equation}
The condition $\theta < \frac{1}{2\kappa_M}$ implies $4\kappa_M\theta < 2$, and thus
\begin{equation*}
q < \frac{\kappa_M - 1 + 2}{\kappa_M + 1} = 1.
\end{equation*}
It remains to verify that the iterates stay in the neighborhood $U(x^*, \delta)$ where Assumptions apply. Choose $\delta_0 \in (0, \delta]$ sufficiently small such that $C\delta_0 \le \frac{1-q}{2}$.
We proceed by induction. Suppose $\|e_m\|_M \le \delta_0$. Then~\eqref{ineq:convergence} implies
\begin{equation*}
\|e_{m+1}\|_M \le q\|e_m\|_M + C\|e_m\|_M^2 = (q + C\|e_m\|_M)\|e_m\|_M.
\end{equation*}
Using the bound on $\delta_0$, we have $q + C\|e_m\|_M \le q + \frac{1-q}{2} = \frac{1+q}{2} < 1$. Thus,
\begin{equation*}
\|e_{m+1}\|_M \le \frac{1+q}{2} \|e_m\|_M \le \delta_0.
\end{equation*}
This confirms that $x_{m+1} \in U(x^*, \delta)$. By induction, if $\|e_0\|_M \le \delta_0$, the sequence remains bounded and converges linearly to $x^*$ with asymptotic rate $\frac{1+q}{2}$.
\end{proof}

Theorem~\ref{thm:convergence} makes clear how the local convergence rate depends on two distinct effects: the conditioning of the preconditioned Hessian and the inexactness of the computed unstable eigenspace. In the ideal case $\theta=0$, the contraction factor reduces to $q\approx\frac{\kappa_M - 1}{\kappa_M + 1}$, which is precisely the preconditioned analogue of the Euclidean HiSD rate. More generally, the term proportional to $\theta$ quantifies the deterioration caused by terminating the inner eigensolver after finitely many steps.

Although Theorem~\ref{thm:convergence} assumes a fixed SPD metric $M$, the result naturally extends to a varying metric $M_m := M(x_m)$ provided $M(x)$ is Lipschitz continuous and uniformly SPD. In that case, the analysis can be performed in the frozen norm $\| \cdot \|_{M(x^*)}$, and the variation of the metric contributes only higher-order terms to the recurrence, preserving the local linear convergence rate.

The local linear estimate also yields the standard complexity bound. Ignoring the higher-order term, the effective contraction factor is $q \approx \frac{\kappa_M-1}{\kappa_M+1}$. Therefore, to achieve $\|x_m-x^*\|_M \le \epsilon \|x_0-x^*\|_M$, it suffices to take $m \ge \frac{\log(1/\epsilon)}{\log(1/q)}$. For $\kappa_M\gg 1$, we have $\log(1/q)\approx 2/\kappa_M$, and hence $m = O\!\left(\kappa_M \log\frac{1}{\epsilon}\right)$. Thus, the discrete p-HiSD iteration retains the same complexity form as the Euclidean scheme, but with the original condition number replaced by the preconditioned one.

\section{Preconditioner Design}
\label{sec:preconditioners}

The convergence theory in Sections~\ref{sec:stability}--\ref{sec:discrete} shows that the performance of p-HiSD depends on how effectively the preconditioner improves the conditioning of the generalized Hessian while remaining computationally manageable. In practice, a useful preconditioner should satisfy three basic requirements: it should reduce the effective condition number $\kappa_M$, remain symmetric positive definite throughout the iteration, and admit an application cost commensurate with the problem scale. In this section, we discuss several representative choices illustrating different trade-offs between spectral accuracy, robustness, and computational cost.

\subsection{Spectral-type preconditioners}

The most direct strategy for improving conditioning is to construct $M$ using the exact spectral information of the Hessian. Assuming the eigendecomposition $H(x)=Q\Lambda Q^\top$ with $\Lambda=\diag(\lambda_1,\ldots,\lambda_n)$, a natural formulation is the spectral preconditioner:
\begin{equation}
\label{eq:spectral-precond}
M_{\mathrm{spec}}(x) = Q \diag(|\lambda_1|+\varepsilon,\ldots,|\lambda_n|+\varepsilon) Q^\top,
\end{equation}
where $\varepsilon>0$ serves as a small regularization parameter. By construction, $M_{\mathrm{spec}}(x)$ is SPD, shares eigenvectors with $H$,
and has generalized eigenvalues
$\sigma_i=\lambda_i/(|\lambda_i|+\varepsilon)$. If
$\mu\le |\lambda_i|\le L$, then
\begin{equation*}
\kappa_{M_{\mathrm{spec}}}
\le \frac{L(\mu+\varepsilon)}{\mu(L+\varepsilon)}
= \frac{1+\varepsilon/\mu}{1+\varepsilon/L}
\to 1 \quad \text{as } \varepsilon\to0^+.
\end{equation*}
Because this bound approaches $1$ as $\varepsilon\to0^+$, the spectral preconditioner is near-optimal in terms of conditioning.

However, this theoretical ideal is fundamentally limited by its computational cost. Recomputing a full eigendecomposition at every step is prohibitive except for small dense problems. To circumvent this bottleneck, a practically useful variant is the frozen spectral preconditioner, in which $M_{\mathrm{spec}}(x_0)$ or a similarly chosen reference matrix is assembled once and then reused over multiple outer iterations. Although this sacrifices exact spectral matching, it typically preserves much of the conditioning benefit while greatly reducing setup cost.

A related geometric challenge arises when the Hessian eigenvectors rotate rapidly, for example near clustered eigenvalues or eigenvalue crossings. In such situations, rigidly updating the spectral frame can lead to discontinuous updates in the preconditioner. To reduce this sensitivity, we introduce a subspace-inertial variant inspired by LOBPCG~\cite{knyazev2001toward} and spectral deflation~\cite{saad2011numerical}. Let $\{(\lambda_i,\tilde v_i)\}_{i=1}^k$ be the $k$ smallest eigenpairs of $H(x)$, and let $V^{\mathrm{old}}=[v_1^{\mathrm{old}},\dots,v_k^{\mathrm{old}}]$ denote the frame from the previous iteration. We first form intermediate vectors:
\begin{equation}
\label{eq:inertial_update}
w_i=(1-\alpha)\sigma_i\tilde v_i+\alpha v_i^{\mathrm{old}},
\qquad
\sigma_i=\operatorname{sign}(\langle \tilde v_i,v_i^{\mathrm{old}}\rangle).
\end{equation}
Here, $\alpha\in[0,1)$ is an inertia parameter. We then orthonormalize these intermediate vectors to obtain the stabilized frame $V=[v_1,\dots,v_k]$. Using this frame, we define the subspace-inertial preconditioner as
\begin{equation}
\label{eq:SI-precond}
M_{\mathrm{SI}}(x) = \mu_{\mathrm{rest}}I+\sum_{i=1}^k(\mu_i-\mu_{\mathrm{rest}})v_i v_i^\top,\end{equation}where $\mu_i=a_i|\lambda_i|+\varepsilon$ and $\mu_{\mathrm{rest}}=\beta|\lambda_{k+1}|+\varepsilon$. The role of the inertial mixing is to gracefully suppress spurious frame rotations while retaining the essential spectral scaling effect. This construction is particularly robust when the unstable eigenspace is more reliable than its individual constituent eigenvectors.

\subsection{Exploiting algebraic structure}

For medium- and large-scale problems, fully spectral constructions are often too expensive. In this regime, it is often necessary to deliberately discard exact spectral fidelity, turning instead to the algebraic structure of the Hessian to maintain computational tractability.

When exploiting structure, the simplest choice is scalar Jacobi scaling:
\begin{equation}
\label{eq:jacobi-precond}
M_{\mathrm{Jac}}(x)
=
\diag(|H_{11}(x)|+\varepsilon,\ldots,|H_{nn}(x)|+\varepsilon).
\end{equation}
This preconditioner is exceptionally cheap to assemble and apply, requiring only $O(n)$ storage and work, making it highly effective when the Hessian is close to diagonally dominant. However, its limitation is equally clear: because it rescales solely along individual coordinate directions, it is blind to strong off-diagonal couplings and cannot adequately capture rotated stiff modes.

To naturally address this limitation in problems with a distinct multi-component or coupled PDE architecture, a block generalization is more appropriate. By partitioning $H(x)$ into blocks $H_{ij}\in\mathbb{R}^{n_i\times n_j}$ such that $\sum_i n_i=n$, we define the block Jacobi preconditioner:
\begin{equation}
\label{eq:block-jacobi-precond}
M_{\mathrm{BJac}}(x)
=
\operatorname{blockdiag}(|H_{11}|+\varepsilon I_{n_1},\ldots,|H_{pp}|+\varepsilon I_{n_p}),
\end{equation}
where $|A|=(A^\top A)^{1/2}$ denotes the matrix absolute value. Directly compensating for the geometric information discarded by scalar scaling, this construction captures local anisotropy within each block and thus better approximates local stiff directions. Its overall effectiveness hinges on the assumption that intra-block dynamics outweigh inter-block couplings; otherwise, crucial global geometry is still lost.

For PDE discretizations dominated by elliptic stiffness, a natural choice is the
Laplacian preconditioner. In one dimension, the discretized negative Laplacian is
$A_h=\frac{1}{h^2}\operatorname{tridiag}(-1,2,-1)$.
We define
\begin{equation}
\label{eq:laplacian-precond}
M_{\mathrm{Lap}}
=
A_h+\alpha I,
\qquad \alpha>0 .
\end{equation}
This construction captures the leading mesh-dependent stiffness, since the
largest eigenvalues of $A_h$ typically scale like $O(h^{-2})$, while remaining
sparse and easy to factorize. It is particularly effective when the
ill-conditioning mainly comes from the Laplacian term. For semilinear problems,
a reaction-enhanced variant is
\begin{equation}
\label{eq:laplacian-reaction-precond}
M_{\mathrm{LapR}}(x)
=
A_h+\operatorname{diag}(|d_h(x)|)+\varepsilon I,
\qquad \varepsilon>0 .
\end{equation}
Here $d_h(x)$ denotes the diagonal reaction contribution in the Hessian.

For sparse Hessians arising from PDE discretizations, incomplete factorization emerges as an effective strategy. A shifted incomplete Cholesky preconditioner is obtained by first selecting $\delta>0$ so that
$H_\delta(x):=H(x)+\delta I \succ 0,$
and then computing an incomplete Cholesky factorization
\begin{equation}
\label{eq:mic_def}
M_{\mathrm{IC}}(x)=\tilde L\tilde L^\top \approx H_\delta(x)
\end{equation}
under a prescribed sparsity pattern or drop tolerance. The shift guarantees positive definiteness even near indefinite saddle points, while the factorization preserves sparsity and yields an application cost proportional to $\mathrm{nnz}(\tilde L)$, where $\mathrm{nnz}(\cdot)$ denotes the number of non-zero elements. For large sparse systems, this often strikes a good balance between retaining curvature information and maintaining computational tractability. In small dense settings, this approach naturally reduces to a full shifted Cholesky factorization.

\subsection{Practical selection guidelines}

The preceding discussion establishes a clear hierarchy for selecting $M$, grounded in the fundamental trade-off between spectral accuracy and computational overhead. Rather than isolated alternatives, these choices represent strategic points along a spectrum from high-fidelity spectral reconstructions to low-cost algebraic approximations.

Table~\ref{tab:precond_selection} summarizes the implementation guidelines based on the problem scale $n$ and sparsity $\rho = \mathrm{nnz}(H)/n^2$. The appropriate choice depends on the trade-off between spectral fidelity and application cost: for small dense systems, spectral information is affordable and beneficial; for larger or sparse problems, structured algebraic preconditioners are more practical.
\begin{table}[htbp]
\centering
\caption{Practical selection guidelines for the p-HiSD preconditioner.}
\label{tab:precond_selection}
\small
\begin{tabular}{ll}
\hline
\textbf{Heuristic Criteria} & \textbf{Recommended Preconditioner $M$} \\ \hline
$n \le 200$ (Small \& Dense) & Spectral / Frozen Spectral \\
$n \le 1000$ (Block structured) & Block Jacobi \\
$\rho < 0.01$ (Sparse) & Shifted Incomplete Cholesky \\
General / Large-scale & Jacobi Scaling (Baseline) \\ \hline
Laplacian-dominated PDE discretization & Laplacian / Reaction-enhanced Laplacian \\
\hline
\end{tabular}
\end{table}

\section{Numerical Experiments}
\label{sec:numerical}

In this section, we assess p-HiSD on a sequence of numerical experiments.
Unless stated otherwise, all methods share the same initial state and stopping tolerance; initial frames are the lowest Hessian eigenvectors for HiSD and the lowest generalized eigenvectors of $(H,M)$ for p-HiSD.
\subsection{Quadratic model: rate verification}
\label{subsec:rate_verification}

We first verify the local convergence rate predicted by Theorem~\ref{thm:convergence} by computing the index-1 saddle point at the origin for the quadratic model
\begin{equation*}\label{eq:quadratic_model}
E(\mathbf{x}) = \tfrac{1}{2}\mathbf{x}^\top H \mathbf{x}, \quad H = \diag(-1, 2, 3, \dots, 100).
\end{equation*}
This yields $n=100$ and spectral bounds $\mu=1$ and $L=100$. We compare standard HiSD, for which $\kappa=100$, with p-HiSD using two target preconditioned condition numbers $\kappa_M \in \{2.0, 1.01\}$. For p-HiSD, we construct a diagonal preconditioner $M_{ii} = |\lambda_i| + \varepsilon$, with $\varepsilon$ chosen so as to realize the prescribed value of $\kappa_M$.

The state step size is chosen as $\eta = 2/(L_M + \mu_M)$ in accordance with Theorem~\ref{thm:convergence}, where $\mu_M$ and $L_M$ are the effective bounds of the generalized eigenvalue problem. Since both $H$ and $M$ are diagonal, initializing the frame with the exact unstable direction $v_1 = e_1$ keeps it fixed throughout the iteration. This isolates the state convergence by eliminating frame approximation errors ($\theta=0$). To exhibit the theoretical worst-case linear rate, all runs are initialized from $\mathbf{x}_0 \propto e_1 + e_{100}$, normalized so that $\|\nabla E(\mathbf{x}_0)\| = 1.0$, and are terminated when $\|\nabla E(\mathbf{x}_m)\| < 10^{-8}$.

\begin{figure}[htbp]
    \centering
    \includegraphics[width=0.55\textwidth]{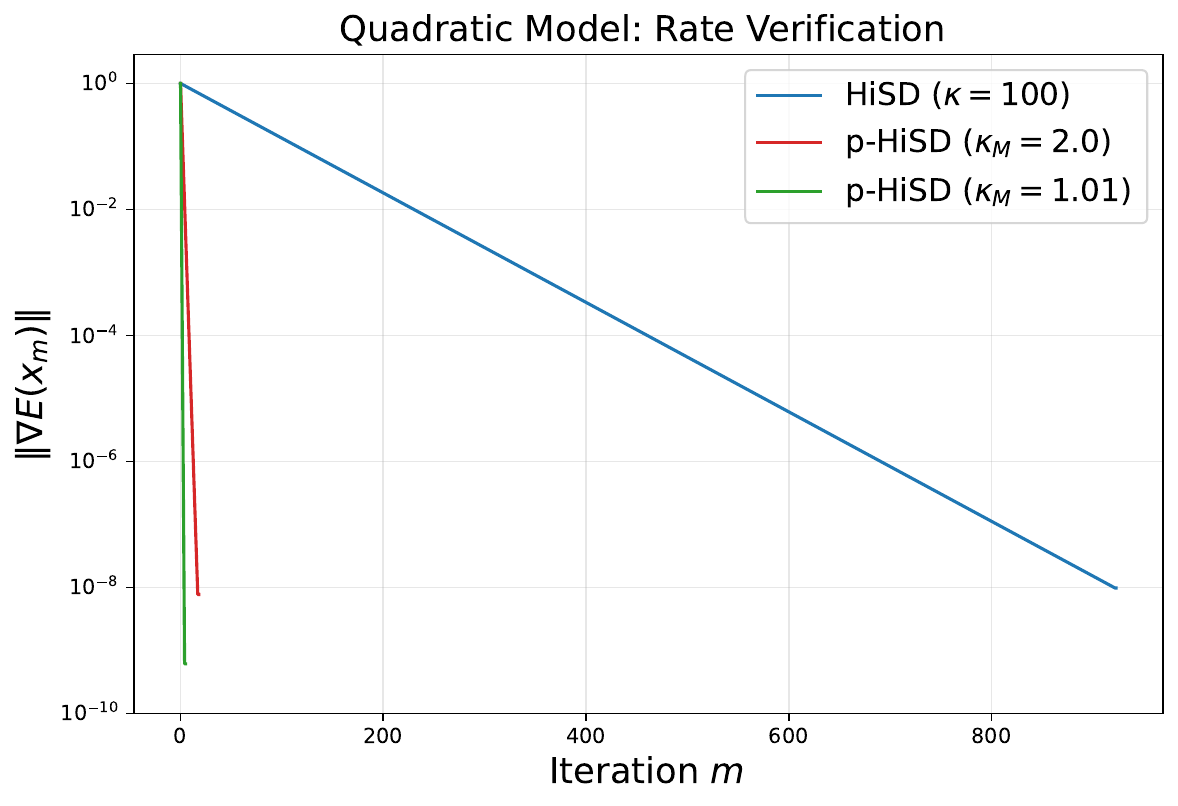}
    \caption{Verification of convergence rate on a quadratic model. The observed rates closely match the theoretical prediction $q = (\kappa_M-1)/(\kappa_M+1)$, showing that preconditioning accelerates convergence by reducing the effective condition number.}
    \label{fig:1}
\end{figure}

The results are shown in Figure~\ref{fig:1}. Standard HiSD requires 923 iterations, with an observed linear rate of approximately $0.980$, in agreement with the factor $(\kappa-1)/(\kappa+1)$. In contrast, p-HiSD with $\kappa_M=2.0$ and $\kappa_M=1.01$ exhibits rates of $0.333$ and $0.005$, converging in only 18 and 5 iterations, respectively. These observations agree closely with the prediction of Theorem~\ref{thm:convergence} and are consistent with the $O(\kappa_M\log(1/\epsilon))$ iteration complexity predicted by the local analysis. 
\subsection{Two-dimensional test problems}
\label{subsec:two_dimensional_tests}
We next use two two-dimensional nonconvex landscapes to illustrate the geometric effect of preconditioning.
\subsubsection{Butterfly function: improved geometric robustness}
We next consider saddle search on the two-dimensional butterfly function \cite{Su2025improved}
\begin{equation*}\label{eq:butterfly}
E(x,y)=x^4-2x^2+y^4+y^2-\tfrac{3}{2}x^2y^2+x^2y-cy^3,
\end{equation*}
with parameter $c=1$.  This landscape features both local minimizers and index-1 saddle points, connected by prominent ridge and valley structures. Such geometry is challenging for standard HiSD: trajectories started near a minimizer often drift away, escaping along energy ridges or missing intermediate saddle points due to misalignment between the computed unstable direction and the true unstable manifold.

To improve this alignment, we apply p-HiSD with two metric choices: the spectral preconditioner \eqref{eq:spectral-precond} and the subspace-inertial preconditioner \eqref{eq:SI-precond}. Since the model is only two-dimensional, we use the exact Hessian at every step, without freezing, so that the comparison isolates the geometric effect of preconditioning from approximation errors in the metric.

We search for the index-$1$ saddle point starting from $(1.44,-0.95)$, a point near the local minimizer. All runs utilize $\eta=0.01$ for the state update, $\tau=0.05$ for the frame update, and $J=3$ inner iterations per outer step. Both p-HiSD variants use regularization parameter $\varepsilon=10^{-2}$, and the subspace-inertial method uses inertia coefficient $\alpha=0.7$ and weight $a_1=0.49$. The stopping criterion is $\|\nabla E(x_m,y_m)\|<10^{-6}$ together with the requirement that the final Hessian have exactly one negative eigenvalue. For visualization, the spectral-preconditioned trajectory is displayed from the symmetric point $(-1.44,-0.95)$.

\begin{figure}[htbp]
    \centering
    \includegraphics[width=0.55\textwidth]{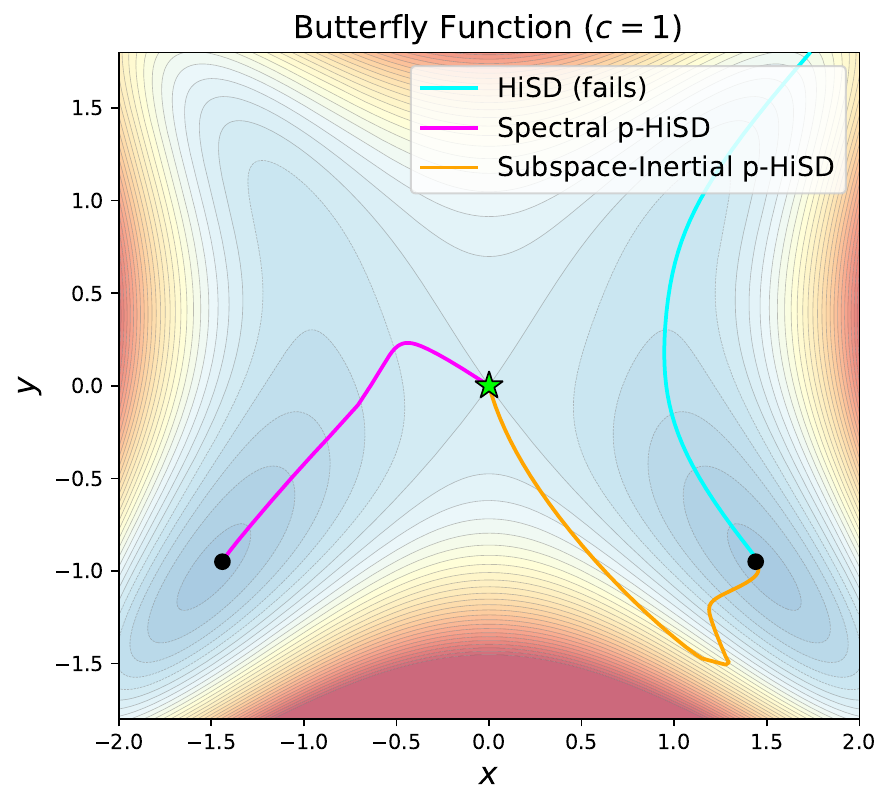}
    \caption{Butterfly function ($c=1$). Starting near a local minimizer, standard HiSD (\textcolor{cyan}{cyan}) fails to reach the target saddle point (\textcolor{green}{green}), whereas p-HiSD with the spectral preconditioner (\textcolor{magenta}{magenta}) and the subspace-inertial preconditioner (\textcolor{orange}{orange}) successfully converges to it.}
    \label{fig:2.1}
\end{figure}

Figure~\ref{fig:2.1} shows the resulting trajectories. The standard HiSD trajectory drifts away from the saddle region and fails to reach the target saddle from this initialization, consistent with the sensitivity discussed in~\cite{Su2025improved}. In contrast, both preconditioned variants converge successfully to the target index-$1$ saddle. This example illustrates that, even in a low-dimensional setting, changing the metric can substantially improve the directional quality of the search and enlarge the basin from which the desired saddle can be reached.

\subsubsection{Modified M\"uller--Brown potential: target-saddle hit statistics}

Following the perturbed benchmark in~\cite{bonfanti2017methods}, we measure
how often the dynamics reaches a prescribed target saddle on the modified
M\"uller--Brown potential
\begin{equation*}
E(x,y)=E_{\rm MB}(x,y)
+
500\sin(xy)\exp\{-0.1(x+0.5582)^2-0.1(y-1.4417)^2\},
\end{equation*}
where $E_{\rm MB}$ denotes the classical M\"uller--Brown potential with the
standard parameters specified in~\cite{bonfanti2017methods}. The localized
perturbation introduces additional curvature variation and provides a useful
test of metric-dependent basin preference.

We compare standard HiSD with p-HiSD equipped with the subspace-inertial
preconditioner \eqref{eq:SI-precond}, abbreviated as SI p-HiSD. The prescribed
index-$1$ saddle is $x^\ast=(0.0660189287,\,0.1840409406)$. We sample $500$
initial points uniformly from $[-0.8,-0.5]\times[1.2,1.5]$, using the same
ensemble for both methods. Both methods use $\tau=10^{-3}$ and $J=5$ inner iterations; the state step sizes are $\eta=10^{-4}$ for HiSD and
$\eta=5\times10^{-2}$ for SI p-HiSD, with the smaller HiSD step chosen for stability.

A run is counted as a target hit if the final iterate satisfies
$\|\nabla E(x_m,y_m)\|<10^{-6}$, lies within distance $10^{-2}$ of $x^\ast$,
and has Morse index one. The index-one condition is checked by
$\lambda_1<-10^{-3}$ and $\lambda_2>10^{-3}$. A converged index-$1$ endpoint
outside the target radius is classified as another index-$1$ saddle.

\begin{table}[htbp]
\centering
\label{tab:2}
\begin{tabular}{c|ccc}
\hline
Method & Target saddle & Other index-$1$ saddle & Other outcome \\
\hline
HiSD & 3/500 (0.6\%) & 497/500 (99.4\%) & 0/500 (0.0\%) \\
SI p-HiSD & 441/500 (88.2\%) & 9/500 (1.8\%) & 50/500 (10.0\%) \\
\hline
\end{tabular}
\caption{Endpoint statistics for the modified M\"uller--Brown potential over
$500$ identical initial conditions. The table reports target-saddle hit
statistics rather than convergence to an arbitrary index-$1$ saddle.}
\end{table}

Table~\ref{tab:2} shows that standard HiSD reliably converges to an index-$1$
saddle, but its trajectories are overwhelmingly attracted to a non-target saddle:
only $3$ out of $500$ runs reach the prescribed saddle. In contrast, SI p-HiSD
reverses this basin preference, reaching the target saddle in $441$ out of
$500$ runs while converging to another index-$1$ saddle in only $9$ runs. Thus,
the target hit rate increases from $0.6\%$ to $88.2\%$, demonstrating that the
preconditioned metric improves target-saddle selectivity.

\subsection{Modified Rosenbrock problem: index-5 saddle search}
\label{subsec:modified_rosenbrock}
Following \cite{Luo2022Convergence}, we consider a modified Rosenbrock
benchmark with a known index-$5$ saddle to test whether the $M$-orthogonal
frame formulation remains effective in a high-dimensional, strongly
anisotropic non-PDE landscape. The energy is
\begin{equation*}
E(x)
=
\sum_{i=1}^{d-1}
\left[
100(x_{i+1}-x_i^2)^2+(1-x_i)^2
\right]
+
\sum_{i=1}^{d}
s_i \bigl(\arctan(x_i-1)\bigr)^2,
\qquad x\in\mathbb{R}^d .
\end{equation*}
Here $d=1000$, with $s_i=-5\times 10^4$ for $1\le i\le 5$ and
$s_i=1$ otherwise. The target point is $x^*=(1,\ldots,1)$. Both terms are stationary at $x^\ast$, and a direct Hessian computation gives
exactly five negative eigenvalues, so $x^\ast$ is a known index-$5$ saddle.

We compare standard HiSD with three p-HiSD variants, using frozen preconditioners
constructed at the initial point: the spectral~\eqref{eq:spectral-precond}, Jacobi~\eqref{eq:jacobi-precond}, and block Jacobi
preconditioners~\eqref{eq:block-jacobi-precond}. The block Jacobi preconditioner uses blocks of size $20$.

All methods start from the same initial point
$x_0=x^\ast+0.1\,\xi/\|\xi\|_2$, where $\xi$ is a standard Gaussian random
vector generated with a fixed seed. We set $\tau=10^{-5}$ and use $J=5$ inner
frame updates per outer iteration for all methods. The p-HiSD variants use
$\eta=0.2$, whereas standard HiSD uses $\eta=10^{-5}$, since increasing the
standard HiSD step size to $\eta=10^{-4}$ leads to rapid divergence. The iteration budget is $10^4$, and a run is counted as converged if  $\|\nabla E(x_m)\|_2<10^{-6}$ and $\|x_m-x^\ast\|_2<10^{-2}$.

\begin{figure}[htbp]
    \centering
    \includegraphics[width=0.55\textwidth]{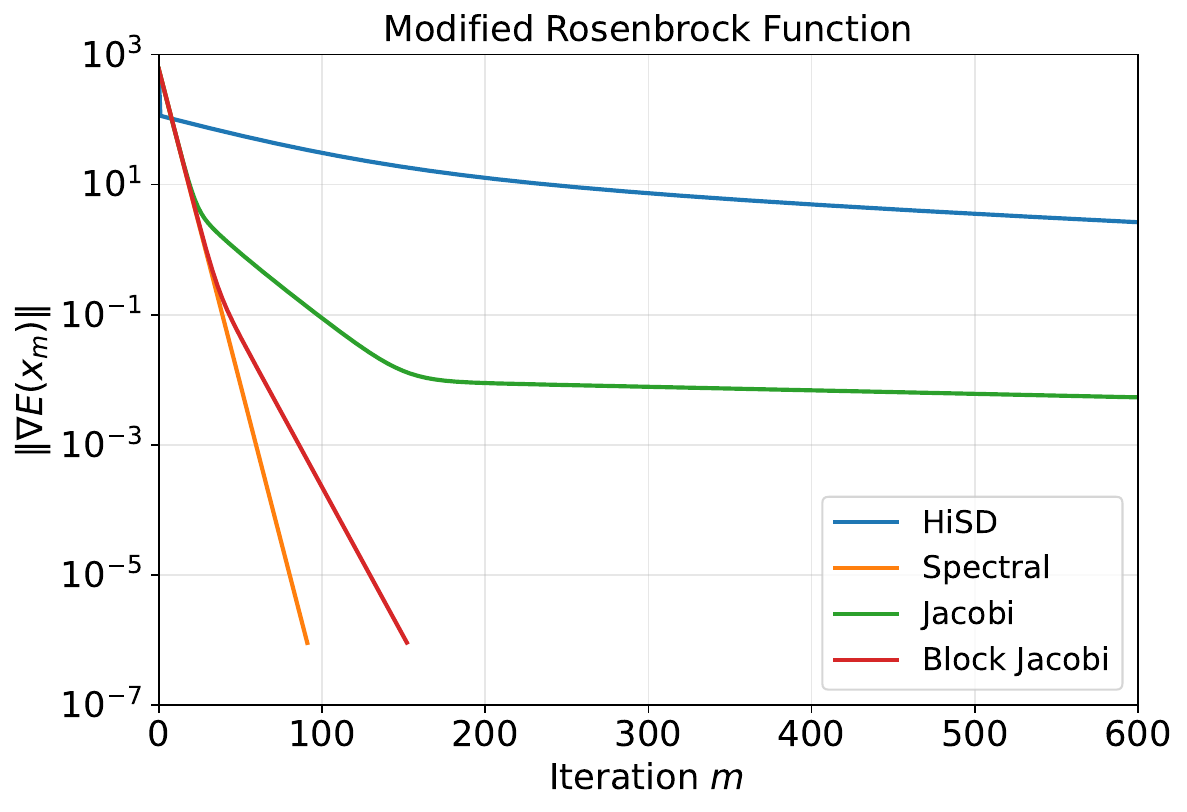}
    \caption{Residual histories for the modified Rosenbrock problem. Only the first
    $600$ iterations are shown; stopping statistics use a $10^4$-iteration budget.}     
    \label{fig:3}
\end{figure}

Figure~\ref{fig:3} shows the residual histories for this fixed initial
perturbation of the known saddle. Standard HiSD does not reach the
stopping tolerance within the $10^4$-iteration budget and attains a final
residual of $1.49\times 10^{-3}$. In contrast, the frozen spectral and
block Jacobi p-HiSD variants converge in 92 and 153 iterations,
respectively. The scalar Jacobi metric, used here as a low-cost diagonal
baseline, also reaches the stopping tolerance, but only after 7506
iterations. The final Hessian index is verified to be five for all
converged p-HiSD runs. These results highlight the effect of preconditioning in this anisotropic
high-index test: all p-HiSD variants reach the prescribed tolerance, whereas
standard HiSD remains above it within the $10^4$-iteration budget. The spectral
and block Jacobi metrics yield the largest reductions in iteration count, while
the scalar Jacobi metric, although slower, still gives a clear improvement over
standard HiSD as a low-cost baseline.

\subsection{Stiff coupled bistable chain}
\label{subsec:diatomic_chain}

We next consider a one-dimensional lattice model consisting of $N$ coupled bistable units $\{(u_i,v_i)\}_{i=1}^N$, where $u_i,v_i\in\mathbb{R}$ denote the internal state variables associated with the $i$-th lattice site. The global state is $\mathbf{x}=[u_1, v_1, \dots, u_N, v_N]^\top \in \mathbb{R}^{n}$ with $n=2N$, and the energy is given by
\begin{equation*}\label{eq:diatomic_energy}
E(\mathbf{x}) = \underbrace{\sum_{i=1}^N \frac{K}{2}(u_i-v_i)^2}_{\text{stiff internal synchronization}}
+ \underbrace{\sum_{i=1}^N \left[(u_i^2-1)^2+(v_i^2-1)^2\right]}_{\text{local bistability}}
+ \underbrace{\sum_{i=1}^{N-1} \frac{\delta}{2}(v_i-u_{i+1})^2}_{\text{weak neighbor coupling}}.
\end{equation*}
When $K\gg\delta$, the Hessian spectrum separates into large eigenvalues from the stiff internal modes $(u_i-v_i)$ and smaller ones from collective modes, making this a natural preconditioning test.

We compare standard HiSD with p-HiSD using block Jacobi \eqref{eq:block-jacobi-precond}, incomplete Cholesky \eqref{eq:mic_def}, and frozen spectral preconditioning based on \eqref{eq:spectral-precond}. We take $N=50$ ($n=100$), $K=10^4$, and $\delta=1$; the block Jacobi metric uses the $50$ local $2\times2$ pairs $(u_i,v_i)$. The index-$1$ search starts from $u_i=v_i=(-1)^i$ perturbed by Gaussian noise of standard deviation $0.1$.

The stiffness parameter $K$ restricts the state step size of standard HiSD: we use $\eta=5\times 10^{-5}$, since increasing it to $5\times 10^{-4}$ leads to divergence in this test. In contrast, all p-HiSD variants use $\eta=0.5$ and remain stable, indicating that the preconditioners rescale the stiff part of the spectrum.

\begin{figure}[htbp]
\centering
\includegraphics[width=0.55\textwidth]{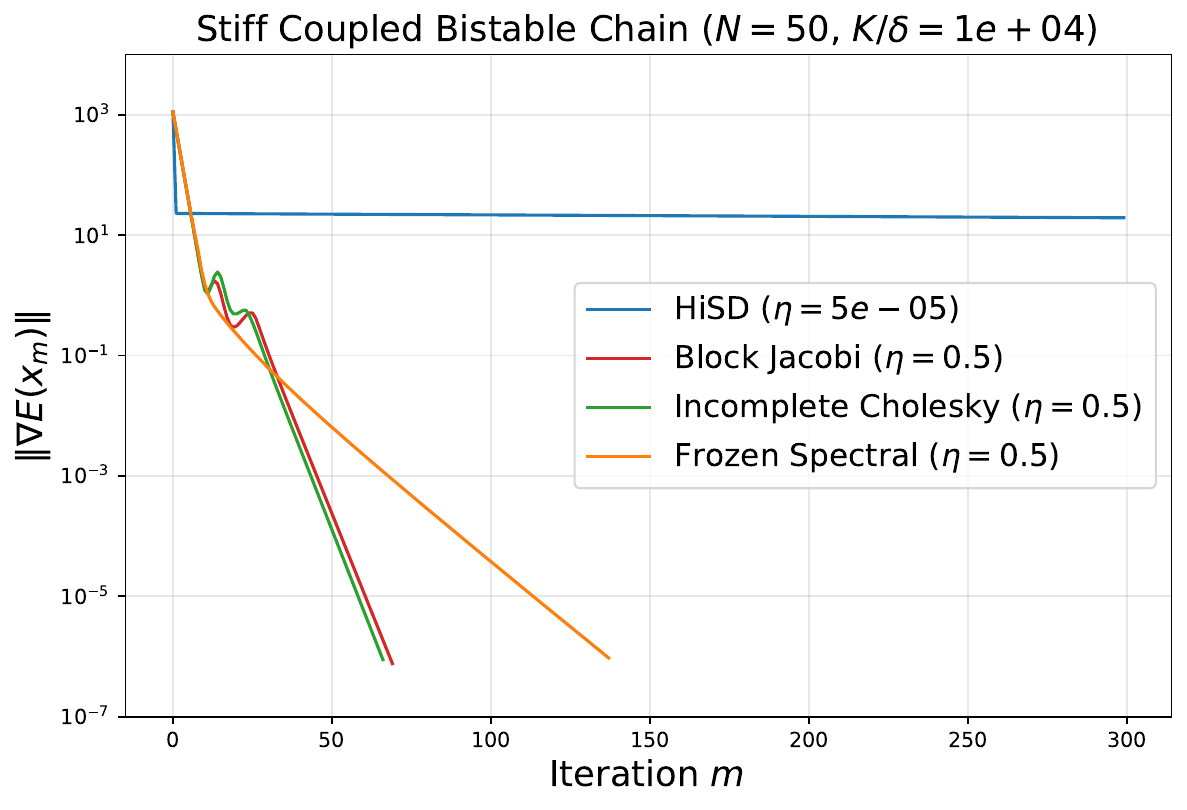}
\caption{Convergence histories of $\|\nabla E(\mathbf{x}_m)\|$ for the coupled bistable chain. Only the first $300$ iterations are plotted; the stopping statistics reported in the text use a $2000$-iteration budget.}
\label{fig:4}
\end{figure}
Figure~\ref{fig:4} shows that standard HiSD stagnates under its restrictive step size, with final gradient norm $9.19$ after the $2000$-iteration budget. In contrast, incomplete Cholesky, block Jacobi, and frozen spectral p-HiSD converge in $67$, $70$, and $138$ iterations, respectively, all reaching residuals below $10^{-6}$. These results show that resolving the local stiffness is essential for efficient saddle search in this multiscale lattice model.

The same example also illustrates the practical selection guideline from Table~\ref{tab:precond_selection}. Using $\eta=0.5$ throughout, we tested three problem sizes:
\begin{center}
\small
\setlength{\tabcolsep}{8pt}
\begin{tabular}{@{}cclcc@{}}
\toprule
Sites ($N$) & DoF ($n$) & Selected Method & Iterations & Time (s) \\
\midrule
50  & 100  & Frozen Spectral     & 138 & 0.031 \\
200 & 400  & Block Jacobi        & 62  & 0.078 \\
501 & 1002 & Incomplete Cholesky & 49 & 0.096 \\
\bottomrule
\end{tabular}
\end{center}

For the smallest system, frozen spectral preconditioning remains affordable and effective. At intermediate scale, block Jacobi offers a better balance between setup and iteration cost. For the largest sparse problem, incomplete Cholesky becomes the most effective choice. Together, these results support the practical usefulness of the selection guideline.

\subsection{Laplacian-dominated PDE discretizations}
\label{subsec:laplacian_dominated_pde}
We next consider three PDE discretizations whose Hessians are dominated by Laplacian or elliptic stiffness. In each case, ill-conditioning is driven by the mesh-dependent spectrum of the underlying differential operator.

\subsubsection{One-dimensional semilinear elliptic problem}
\label{subsec:semilinear_elliptic}

The first example is a one-dimensional semilinear elliptic problem on $\Omega=(0,\pi)$,
\begin{equation*}
    u_{xx}+u^4-10u^2=0,\qquad x\in(0,\pi),\qquad u(0)=u(\pi)=0.
\end{equation*}
This equation is the Euler--Lagrange equation of
\begin{equation*}
    E(u)=\int_0^\pi \left(\frac12 |u_x|^2-F(u)\right)\,dx,
    \qquad F'(u)=u^4-10u^2.
\end{equation*}
Using second-order finite differences with homogeneous Dirichlet boundary
conditions, the finite-dimensional gradient is
$\nabla E_h(u)=A_hu-(u^4-10u^2)$, where $A_h$ is the positive Dirichlet
Laplacian and the nonlinear term is applied componentwise. We search for an index-$3$ saddle in this stiff discretized
problem.

We compare standard HiSD with four p-HiSD variants using the spectral
preconditioner \eqref{eq:spectral-precond}, the block Jacobi preconditioner
\eqref{eq:block-jacobi-precond} with block size $64$, the shifted incomplete
Cholesky preconditioner \eqref{eq:mic_def}, and the reaction-enhanced
Laplacian preconditioner \eqref{eq:laplacian-reaction-precond}. For the reaction-enhanced Laplacian metric, we use the frozen form
$M_{\mathrm{LapR}}(u)=A_h+\operatorname{diag}(|4u^3-20u|)+\varepsilon I$.

For the numerical experiment, we use $N=400$ interior grid points and initialize
the iteration from $u_0(x)=\sin(4x)$. For standard HiSD, we take
$\eta=10^{-5}$, $\tau=0.01$, and $J=5$, since $\eta=10^{-4}$ is unstable for HiSD in this test. For p-HiSD, all variants use $\tau=0.01$, $J=5$, and initial state step size
$\eta=0.25$, with preconditioners first constructed at $u_0$. Once the one-dimensional discrete $L^2$ residual $\|\nabla E_h(u_m)\|_h:=\sqrt{h}\|\nabla E_h(u_m)\|_2$ falls below $3\times10^{-1}$, we rebuild the preconditioner once; after this re-freezing,
we use $\eta=1.5$ for the spectral, shifted incomplete Cholesky, and
reaction-enhanced Laplacian metrics, while the block Jacobi metric uses the
smaller step size $\eta=0.9$ to maintain stability. The iteration is terminated when $\|\nabla E_h(u_m)\|_h < 10^{-6}$.
    
\begin{figure}[htbp]
    \centering
    \includegraphics[width=0.55\textwidth]{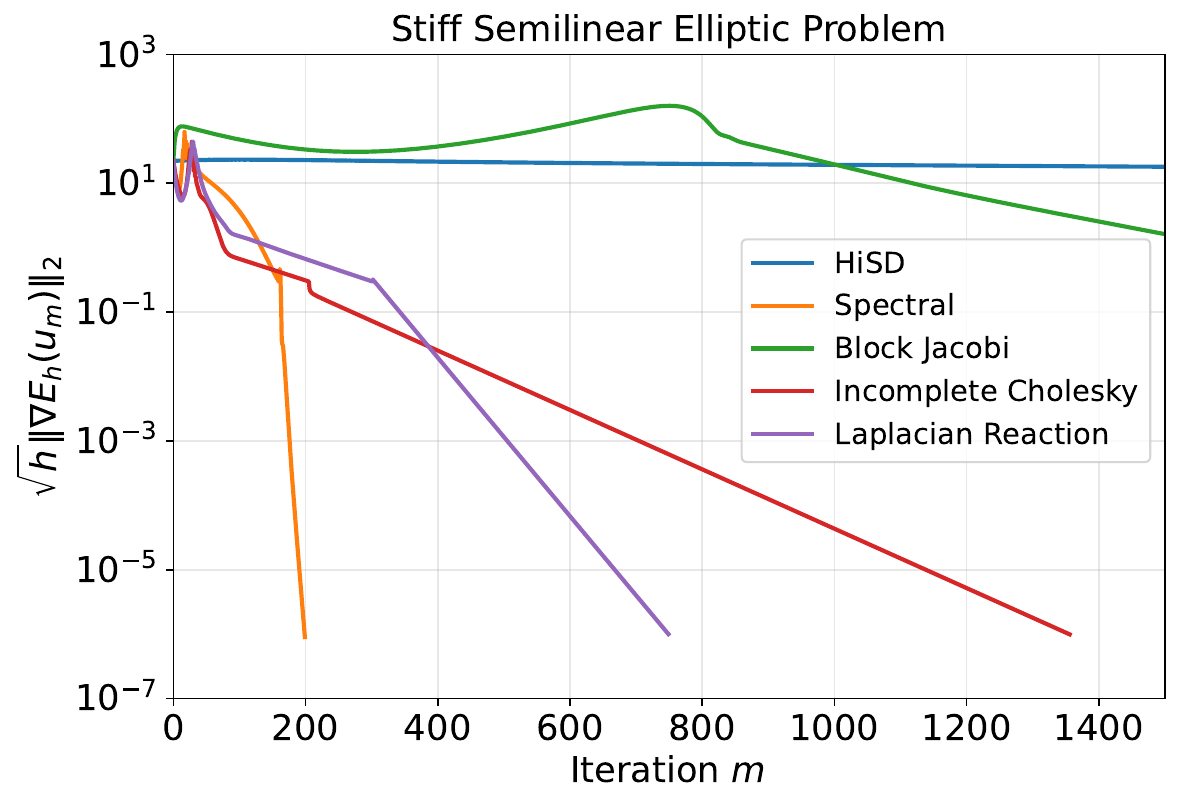}
    \caption{Residual histories for the elliptic problem in the discrete $L^2$
    norm $\|\nabla E_h(u_m)\|_h$. Only the first $1500$ iterations are shown;
    stopping statistics use a $6000$-iteration budget.}
    \label{fig:5.1}
\end{figure}

Figure~\ref{fig:5.1} shows that standard HiSD is restricted by the
elliptic stiffness and remains at a residual of order $10^1$ after the
$6000$-iteration budget. In contrast, all p-HiSD variants converge to an
index-$3$ saddle. The spectral, reaction-enhanced Laplacian, shifted
incomplete Cholesky, and block Jacobi preconditioners require 199, 749,
1356, and 5003 iterations, respectively. These results show that, for this one-dimensional elliptic discretization,
p-HiSD can use much larger stable state step sizes than standard HiSD and reach
the prescribed tolerance with substantially fewer outer iterations. The
one-time re-freezing is an implementation choice in this experiment and is not
part of the fixed-metric convergence theory.

\subsubsection{Two-dimensional Lane--Emden-type elliptic equation}
\label{subsec:lane_emden}

The second example is a two-dimensional Lane--Emden-type equation on
$\Omega=(0,\pi)^2$,
\begin{equation*}
\Delta u + u^p = 0 \quad \text{in } \Omega,
\qquad
u=0 \quad \text{on } \partial\Omega,
\end{equation*}
where $p\in\{3,5\}$. Its solutions are critical points of the energy
\begin{equation*}
E(u)=\int_\Omega \left(\frac12 |\nabla u|^2-\frac{1}{p+1}u^{p+1}\right)\,dx .
\end{equation*}
With $N$ interior grid points in each coordinate direction, the finite-difference
discretization yields $u\in\mathbb{R}^{N^2}$ and
$\nabla E_h(u)=A_hu-u^p$, where $A_h$ is the Dirichlet discrete negative
Laplacian and $u^p$ is applied componentwise. For both $p=3$ and $p=5$,
we target an index-$1$ saddle.

We use this example for two tests: a fixed-grid comparison between standard HiSD
and Laplacian-preconditioned p-HiSD, and a mesh-refinement test for p-HiSD.
Both tests start from the nodal values of $u_0(x,y)=1.2\sin x\sin y$.
The p-HiSD method uses the fixed Laplacian metric
$M_{\mathrm{Lap}}=A_h+\alpha I$ from~\eqref{eq:laplacian-precond},
with $\alpha=1$ and the present two-dimensional $A_h$; $M_{\mathrm{Lap}}$
is factorized once before the iteration and then kept fixed.

For the fixed-grid comparison, we take $N=128$ and stop when $\|\nabla E_h(u_m)\|_h := h\|\nabla E_h(u_m)\|_2 < 10^{-6}$,
where $h=\pi/(N+1)$. All runs use $\tau=10^{-4}$ and $J=5$; the state step is
$\eta=10^{-4}$ for HiSD and $\eta=0.5$ for p-HiSD. The larger HiSD step $\eta=10^{-3}$ is unstable in this test. For mesh refinement, we run p-HiSD on $N=64,128,192,256$ for both
$p=3$ and $p=5$, using the same tolerance and p-HiSD parameters.

\begin{table}[htbp]
\centering
\footnotesize
\setlength{\tabcolsep}{4pt}
\renewcommand{\arraystretch}{0.95}
\caption{Two-dimensional Lane--Emden-type problem. 
Panel (a) compares HiSD and Laplacian-preconditioned p-HiSD on the fixed grid $N=128$.
Panel (b) reports mesh-refinement results for p-HiSD, where 
$r_m=\|\nabla E_h(u_m)\|_h$ denotes the final residual.}
\label{tab:lane-emden-2d}

\setcounter{tablepanel}{0}

\refstepcounter{tablepanel}
\label{tab:lane-emden-fixed}
\begin{tabular}{c c c c c}
\toprule
\multicolumn{5}{c}{\textbf{(a) Fixed grid, $N=128$}} \\
\midrule
$p$ & Method & $\eta$ & Iter. & Final $r_m$ \\
\midrule
$3$ & HiSD   & $10^{-4}$ & $32181$ & $1.00\times 10^{-6}$ \\
$3$ & p-HiSD & $0.5$     & $45$    & $7.41\times 10^{-7}$ \\
$5$ & HiSD   & $10^{-4}$ & $34246$ & $1.00\times 10^{-6}$ \\
$5$ & p-HiSD & $0.5$     & $52$    & $9.12\times 10^{-7}$ \\
\bottomrule
\end{tabular}

\vspace{1mm}

\refstepcounter{tablepanel}
\label{tab:lane-emden-mesh}
\begin{tabular}{c c c c c c}
\toprule
\multicolumn{6}{c}{\textbf{(b) Mesh refinement, p-HiSD}} \\
\midrule
& & \multicolumn{2}{c}{$p=3$} & \multicolumn{2}{c}{$p=5$} \\
\cmidrule(lr){3-4}\cmidrule(lr){5-6}
$N$ & DoF & Iter. & Final $r_m$ & Iter. & Final $r_m$ \\
\midrule
$64$  & $4096$  & $45$ & $7.54\times 10^{-7}$ & $52$ & $9.77\times 10^{-7}$ \\
$128$ & $16384$ & $45$ & $7.41\times 10^{-7}$ & $52$ & $9.12\times 10^{-7}$ \\
$192$ & $36864$ & $45$ & $7.39\times 10^{-7}$ & $52$ & $9.00\times 10^{-7}$ \\
$256$ & $65536$ & $45$ & $7.38\times 10^{-7}$ & $52$ & $8.96\times 10^{-7}$ \\
\bottomrule
\end{tabular}
\end{table}

Table~\ref{tab:lane-emden-fixed} shows that, on the fixed grid $N=128$,
Laplacian-preconditioned p-HiSD reaches the tolerance in $45$ and $52$
iterations for $p=3$ and $p=5$, respectively, compared with $32181$ and
$34246$ iterations for HiSD. In the p-HiSD mesh-refinement test in Table~\ref{tab:lane-emden-mesh}, the
iteration counts remain unchanged for this initial condition and tolerance as
$N$ increases from $64$ to $256$; all reported residuals are below $10^{-6}$,
and the final Hessian index is one. These results are consistent with the Laplacian metric reducing the leading
$h^{-2}$ stiffness: on the fixed grid it permits a much larger stable state
step, while in the mesh-refinement test the p-HiSD iteration counts remain
unchanged across the tested meshes.

\subsubsection{Allen--Cahn equation}
\label{subsec:allen_cahn}

The third example is an index-$1$ saddle-search problem for the Allen--Cahn
energy on the two-dimensional domain $\Omega=(0,1)^2$,
\begin{equation*}
E(u)=\int_{\Omega}\left(\frac{1}{2}|\nabla u|^2+\frac{1}{\xi^2}F(u)\right)\,dx ,
\end{equation*}
where $u\in H^1(\Omega)$, $\xi>0$ is the interfacial width, and $F(u)=\frac14(u^2-1)^2$. With homogeneous Neumann boundary conditions, the critical points satisfy $-\Delta u+\xi^{-2}(u^3-u)=0$; we target the index-$1$ saddle separating the two stable phases $u\approx \pm1$.

For the numerical discretization, we use a uniform $N\times N$ grid with $N=80$ and mesh size $h=\frac{1}{N-1}$. The Laplacian is approximated by the standard five-point finite difference stencil with the usual boundary modifications for the Neumann condition. We set $\xi=0.07$, giving $\xi/h\approx 5.5$, sufficient to resolve the diffuse interface without significant grid pinning. The iteration is terminated when $\|\nabla E(u_m)\|\le 10^{-6}$.

This problem is strongly stiff, and standard HiSD with $M=I$ is stable only for very small step sizes: in our tests, $\eta=10^{-5}$ is admissible, while $\eta=10^{-4}$ already leads to rapid divergence. To overcome this restriction, we use a heuristic two-stage p-HiSD strategy to balance cost and accuracy.

In the initial transient phase, we use a frozen preconditioner $M_1 = -\Delta_h + \frac{2}{\xi^2} I$. This replaces the spatially varying reaction curvature $(3u^2-1)/\xi^2$ by a positive constant surrogate, producing a cheap SPD metric that captures the dominant stiffness of the Laplacian term without tracking the full Hessian geometry. This approximation is sufficient to stabilize the early iterations and allows the use of step size $\eta=0.2$. 

As the trajectory approaches the saddle, the constant shift in $M_1$ becomes too crude to resolve the spatial variation in the reaction curvature. We therefore switch, once the relative decrease of the gradient norm over the last 10 iterations falls below 10\%, to a second frozen shifted-Hessian preconditioner $M_2(u)=H(u)+\sigma I$, where
$H(u)=-\Delta_h+\operatorname{diag}((3u^2-1)/\xi^2)$ and
$\sigma=\max(0,-\lambda_{\min}(H(u)))+0.1$ ensures positive definiteness.

Compared with $M_1$, the metric $M_2(u)$ retains the spatially varying reaction curvature and therefore provides a more faithful approximation of the local Hessian geometry near the saddle. This improved curvature matching allows a larger step size, $\eta=1.0$, in the final convergence phase. Both stages use $\tau=10^{-2}$ and $J=5$ inner iterations for the frame update. Each preconditioner is factorized once via sparse Cholesky decomposition and remains frozen thereafter, so the extra cost appears mainly in the one-time setup rather than in every outer iteration. 

\begin{figure}[htbp]
  \centering
  \includegraphics[width=0.55\textwidth]{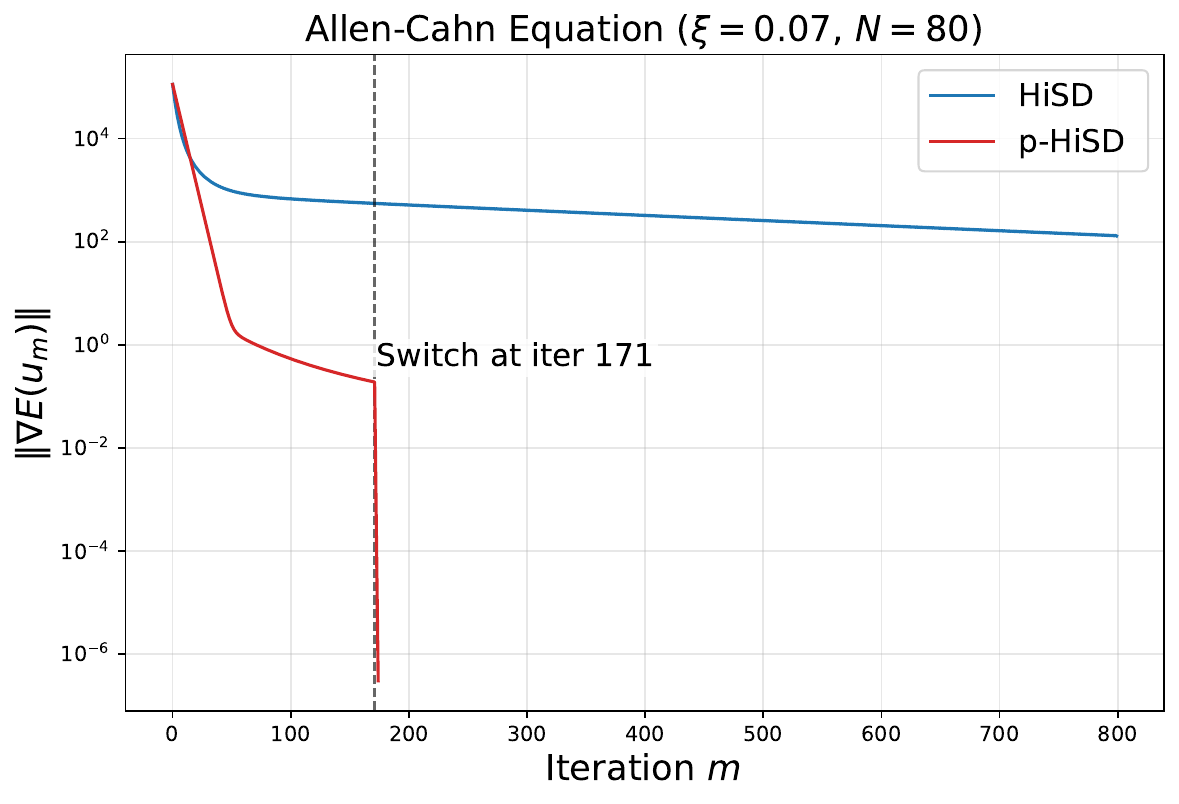}
  \caption{Allen--Cahn: convergence histories of the discrete gradient norm
  $\|\nabla E(u_m)\|$. Standard HiSD decreases slowly under the tested step size,
  whereas the two-stage p-HiSD strategy reaches the prescribed tolerance in
  $175$ iterations, with the preconditioner switched at iteration $171$.}
  \label{fig:5.3}
\end{figure}

Figure~\ref{fig:5.3} compares the convergence histories. Under the step-size
restriction, standard HiSD decreases slowly: after $800$ iterations, the
gradient norm is still approximately $129$. In contrast, the two-stage p-HiSD
strategy reaches the prescribed tolerance in this test. The first stage using
$M_1$ reduces the gradient norm to about $0.19$ within $171$ iterations, and
after switching to $M_2$ the method reaches
$\|\nabla E(u_m)\|\approx 3\times 10^{-7}$ in $4$ additional steps. This
example indicates that a low-cost approximate metric can be useful in the
transient regime, while a curvature-aware metric improves the local
preconditioning near the saddle.

\subsection{Non-convex optimal control of an elliptic equation}
\label{subsec:optimal_control}
Finally, we test HiSD and p-HiSD on the reduced functional of a non-convex
elliptic optimal-control problem. Although the control problem is written in
minimization form, the reduced functional $\hat J$ is non-convex in this test;
we therefore use it as a saddle-search benchmark and seek an index-$1$ saddle
in the control space. Let $\Omega=(0,1)$ and consider
\begin{equation*}
\min_{u\in H_0^1(\Omega)} J(y,u)
:=
\frac12 \|y-y_d\|_{L^2(\Omega)}^2
+
\frac{\lambda}{2}\|u\|_{H^1(\Omega)}^2,
\end{equation*}
subject to the nonlinear elliptic state equation
\begin{equation*}
-y''+y+y^3 = g(u) \quad \text{in } \Omega,
\qquad
y=0 \quad \text{on } \partial\Omega .
\end{equation*}
Here $g(u)=0.001u^2+\cos(2\pi u)$, the desired state is
$y_d(x)=-2\sin(\pi x)$, and we take $\lambda=0.02$. For each control $u$, let $y(u)$ solve the state equation and define the
reduced functional $\hat J(u):=J(y(u),u)$.

We discretize $\Omega=(0,1)$ by second-order finite differences on a uniform
grid with $N$ interior points and $h=1/(N+1)$, representing the control by
$u\in\mathbb{R}^N$. Let $A_h$ be the one-dimensional Dirichlet discrete
negative Laplacian and $I$ the identity matrix. 

For each control $u$, we solve the discrete state equation
$A_h y+y+y^3-g(u)=0$ by Newton's method. The reduced gradient is then
evaluated by the adjoint method, with $p$ determined from
$(A_h+I+\operatorname{diag}(3y^2))p=y-y_d$.

Dropping the common quadrature factor $h$, we use
$\nabla \hat J(u)=\lambda(A_h+I)u+p\odot g'(u)$ in the saddle dynamics,
where $\odot$ denotes componentwise multiplication. Reduced Hessian-vector
products are approximated by centered finite differences of this gradient,
with increment $10^{-5}$ along the normalized direction.

We compare standard HiSD with three p-HiSD variants using fixed metrics:
a shifted Cholesky metric and the frozen spectral metric~\eqref{eq:spectral-precond}
constructed from the initial dense reduced Hessian
$H_0=\nabla^2\hat J(u_0)$, and the Laplacian metric
$M_{\mathrm{Lap}}=A_h+I$ from~\eqref{eq:laplacian-precond}. The Cholesky
shift is chosen so that the metric is positive definite.

In the experiment, we take $N=255$ ($h=2^{-8}$) and set $u_0(x)=0.5\sin(2\pi x)$. We seek an index-$1$ saddle and stop when
$\|\nabla \hat J(u_m)\|_2<10^{-6}$. All methods use $\tau=10^{-3}$ and
$J=5$ direction updates per outer step; the state step is $\eta=10^{-4}$
for HiSD and $\eta=1$ for all p-HiSD variants. The larger HiSD step
$\eta=10^{-3}$ is unstable here.

\begin{figure}[htbp]
\centering
\includegraphics[width=0.55\textwidth]{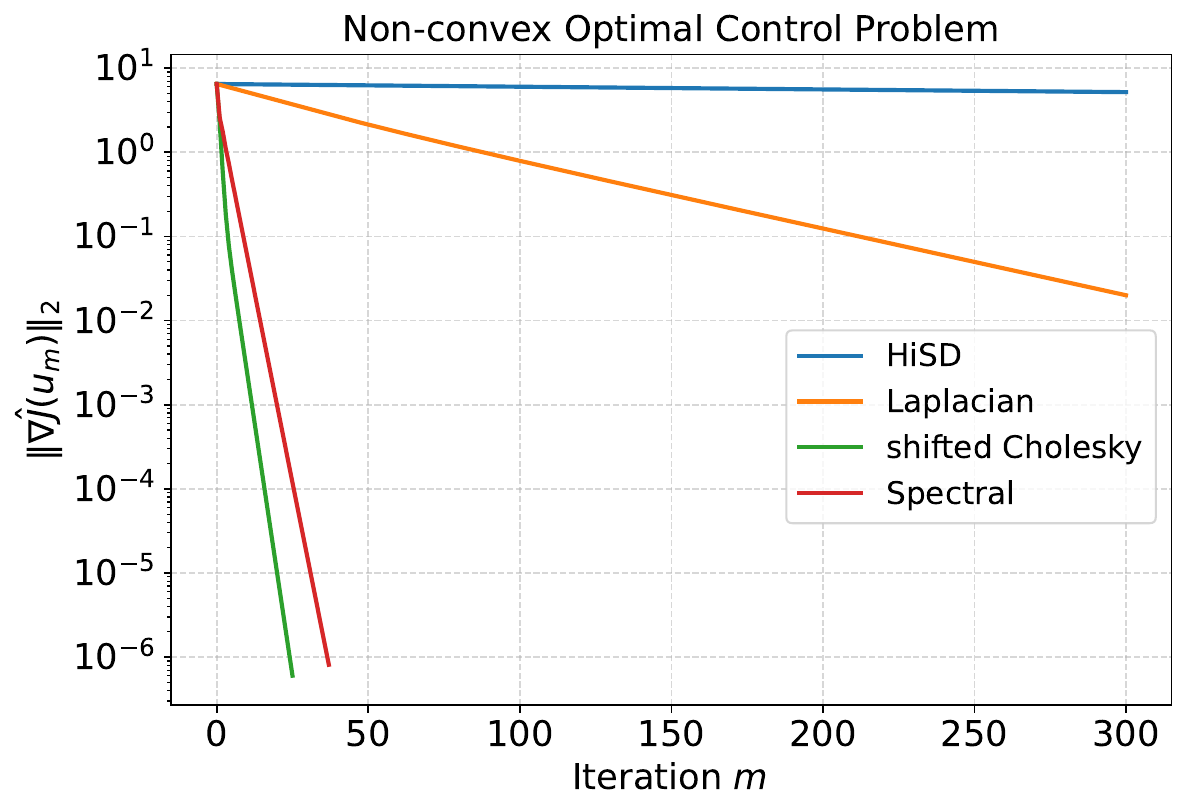}
\caption{Reduced-gradient residuals for the non-convex optimal-control problem.
Only the first $300$ outer iterations are shown; stopping statistics use a
$10^4$-iteration budget.}
\label{fig:6}
\end{figure}

Figure~\ref{fig:6} shows the reduced-gradient residuals. With the
$10^4$-iteration budget, standard HiSD remains above the tolerance
(final residual $1.07\times 10^{-2}$). The shifted Cholesky, frozen spectral,
and Laplacian p-HiSD variants reach $\|\nabla\hat J(u_m)\|_2<10^{-6}$ in
$25$, $37$, and $841$ iterations, respectively, all with final reduced
Hessian index one. These results indicate that, in this reduced
optimal-control test, the two Hessian-based metrics are particularly
effective in iteration count, while the Laplacian metric still reaches the
prescribed tolerance as an operator-based alternative.

\section{Conclusion}
\label{sec:conclusion}
In this work, we developed a preconditioned high-index saddle dynamics (p-HiSD) framework for computing saddle points in a Riemannian metric induced by a symmetric positive definite preconditioner $M$. We established the theoretical foundations of the method by proving the invariance of critical points and their Morse indices under preconditioning (Proposition~\ref{prop:invariance}) and demonstrating local exponential stability by Jacobian spectral analysis (Theorem~\ref{thm:stability}). Furthermore, we derived a discrete linear convergence rate of $(\kappa_M-1)/(\kappa_M+1)$, yielding an iteration complexity of $O(\kappa_M\log(1/\epsilon))$ (Theorem~\ref{thm:convergence}). 
To translate these theoretical gains into practice, we discussed several preconditioners, ranging from spectral and subspace-inertial methods to Jacobi, block Jacobi, and incomplete Cholesky variants, accompanied by a heuristic selection guideline (Table~\ref{tab:precond_selection}).

Numerical experiments on a sequence of test problems support these theoretical findings. 
The quadratic rate-verification test (Section~\ref{subsec:rate_verification}) is consistent with the predicted local convergence rates.
The two-dimensional landscape tests (Section~\ref{subsec:two_dimensional_tests}) illustrate how the preconditioned metric affects search trajectories and endpoint selection. 
In the modified Rosenbrock test (Section~\ref{subsec:modified_rosenbrock}), the tested p-HiSD variants reach the prescribed tolerance, whereas standard HiSD remains above the tolerance within the $10^4$-iteration budget. 
For the stiff coupled bistable chain (Section~\ref{subsec:diatomic_chain}),
the Laplacian-dominated PDE discretizations
(Section~\ref{subsec:laplacian_dominated_pde}), and the reduced optimal-control
problem (Section~\ref{subsec:optimal_control}), the tested preconditioned
variants allow larger stable state steps and require fewer outer iterations
than standard HiSD under the corresponding parameter choices.
The Lane--Emden mesh-refinement test (Section~\ref{subsec:lane_emden})
further shows unchanged p-HiSD iteration counts across the tested grids.
These results suggest that preconditioning can play an important role in accelerating HiSD and improving its practical performance on stiff and ill-conditioned problems.

Future work includes extending this framework to infinite-dimensional settings via operator preconditioning for PDE-constrained optimization, developing more adaptive preconditioners updated along search trajectories, and combining p-HiSD with downward/upward search for automated exploration of complex solution landscapes.

\section*{Code Availability}
\noindent
The source code related to the implementation of the proposed method is available at
\url{https://github.com/hbz1220/p-HiSD}.

\bibliographystyle{siamplain}
\bibliography{references}

\end{document}